\documentclass[preprint,12pt]{elsarticle}
\usepackage{amsfonts}
\usepackage{amssymb}
\usepackage{amsmath,amsthm}
\usepackage{amssymb}

\newtheorem{theorem}{Theorem}[section]

\newtheorem{proposition}[theorem]{Proposition}

\newtheorem{assumption}[theorem]{Assumption}

\newtheorem{example}[theorem]{Example}

\newtheorem{remark}[theorem]{Remark}

\begin{document}

\begin{frontmatter}
\title{Stochastic Burgers PDEs with random coefficients and a generalization of the Cole-Hopf transformation
}

\author[ne]{Nikolaos Englezos}
 \ead{englezos@unipi.gr}
\author[nf]{Nikolaos E. Frangos}
\ead{nef@aueb.gr}
\author[nf]{Xanthi-Isidora Kartala \corref{cor1}}
\ead{xkartala@aueb.gr}
\author[nf]{Athanasios N. Yannacopoulos}
\ead{ayannaco@aueb.gr}

\address[ne]{
Department of Banking and Financial Management,
 University of Piraeus,
 80 Karaoli \& Dimitriou Str.,
18534, Piraeus, Greece.}

\address[nf]{
Department of Statistics,
Athens University of Economics and Business,
76 Patission Str.,
10434, Athens, Greece.}

\address{\footnotesize{\emph{This is the authors' version of a work that was accepted for publication in \emph{Stochastic Processes and Their Applications}. Changes resulting from the publishing process, such as peer review, editing, corrections, structural formatting, and other quality control mechanisms may not be reflected in this document. Changes may have been made to this work since it was submitted for publication.}}}

\begin{abstract}
This paper studies forward and backward versions of random Burgers equation (RBE) with stochastic coefficients. Firstly, the celebrated Cole-Hopf transformation reduces the forward RBE to a forward random heat equation (RHE) that can be treated pathwise.
Next we provide a connection between the backward Burgers equation and a system of FBSDEs. Exploiting this connection, we derive a generalization of the Cole-Hopf transformation which links the backward RBE with the backward RHE and investigate the range of its applicability. Stochastic Feynman-Kac representations for the solutions are provided. Explicit solutions are constructed and applications in stochastic control and mathematical finance  are discussed.
\end{abstract}


\end{frontmatter}

\section{Introduction}\label{sec:intro}
Burgers equation plays a very important r\^ole in the theory of differential equations and applied mathematics. It is a quasilinear parabolic partial differential equation (PDE) of the form
\begin{equation}
U_{t} + U \, U_{x} = \frac{1}{2}\,\nu \, U_{xx}\, ,
\label{eqn:introBURGERS}
\end{equation}
where $U(t,x)$ is an unknown velocity field, which is to be determined by the initial condition $U(0,x)=u_{0}(x),$
and $\nu$ is a viscosity term;
hereafter explicit dependence of the subsequent fields is suppressed for ease of notation and subscripts will denote partial derivatives with respect to the corresponding variable.
This equation can be considered either in the whole real line $(x\in \mathbb{R})$
 or in a bounded interval with Dirichlet or Neumann boundary conditions.
It has been proposed in the 1930's by the Dutch scientist J.M. Burgers as a simple model for the dynamics of the Navier-Stokes equations in one spatial dimension \cite{Burgers1,Burgers2}. As such it has been used to model the dynamics of one dimensional pressureless  turbulence in fluid flows and the hope was that it would help to  understand a lot of the intricate structure of this fascinating subject.
One of the great breakthroughs in this study was the discovery by J. D. Cole \cite{Cole} and E. Hopf \cite{Hopf} in the 1950's of a transformation that reduces this equation to the heat equation, thus allowing the derivation of exact solutions in closed form. In particular by the transformation $U=-\frac{\partial}{\partial x}\ln V$ to a new variable $V(t,x)$, where we take without loss of generality that $\nu=1$, the Burgers equation transforms to the linear heat equation
 \begin{eqnarray}
 V_{t}=\frac{1}{2}\,V_{xx}\, .
 \label{eqn:introHEAT}
 \end{eqnarray}
 This transformation, which has been named the Cole-Hopf transformation\footnote{ It is worth noting that a version of Burgers equation was proposed in 1906 by Forsyth, who in fact proposed a version of the linearising transformation, and a similar equation was further proposed by Bateman in 1915. For the history of the Burgers equation and the Cole-Hopf transformation see \cite{Holden}.}, allows us to construct explicit analytic solutions for the Burgers equation.

This breakthrough ended temporarily the career of Burgers equation as a modeling tool for turbulence, however, it turned it into a scientific
paradigm (a) of a nonlinear PDE that is linearizable with the use of a simple tranformation (in some sense the PDE analogue of the Ricatti equation) and (b) of a benchmark model that can be used to understand the basic features of the interaction between nonlinearity and dissipation. The scientific community soon found many new and interesting applications for the Burgers equation, other than its initial ones in fluid mechanics. For instance, as a model for condensed matter physics, in statistical physics (a continuous time version of ballistic deposition models within the framework of the Kardar-Parisi-Zhang (KPZ) model), in cosmology within what is referred to as the Zel'dovich approximation \cite{Zeldovich}, in the modelling of traffic flow \cite{Chowdhury}, and even in economic theory \cite{Hodges1, Hodges2, Yannacopoulos}.

However, the Burgers equation developed a parallel, theoretical, and abstract mathematical existence beyond its dominant presence in applications. Motivated by the intention to reinstate  Burgers equation as a model for turbulence, the nonlinear dynamics community turned its attention to the \emph{randomly} forced Burgers equation.  This introduced a class of quasilinear \emph{stochastic} PDEs (SPDEs) of the form
\begin{eqnarray}
U_{t} + U \, U_{x} = \frac{1}{2}\, U_{xx} + F(t,x,\omega),
\nonumber
\end{eqnarray}
where now $F(\cdot,x,\cdot)$ is a stochastic process that acts as the forcing term. This turns the solution into a \emph{random field}, whose properties are to be determined by the data of the problem, i.e., the forcing term and the initial condition. Depending on the properties of the random forcing term, this equation has to be treated accordingly, i.e., pathwise or in the It\^o sense
\begin{eqnarray}
dU = \left(\frac{1}{2}\,U_{xx} - U \,U_{x}\right) dt + Q \, dW(t),
\nonumber
\end{eqnarray}
in the case where the forcing term $F(t,x,\omega)$ can be modelled as an infinite dimensional Brownian motion.

The introduction of \emph{randomness} in Burgers equation produced a number of very interesting new directions; directions connected with dynamical systems aspects of the equation, e.g. existence and properties of invariant measures (see for instance the important contributions by E et al \cite{E} or Goldys and Maslowski \cite{Goldys}), directions related to various questions on the well-posedeness of the equation in various functional settings using techniques from infinite dimensional stochastic analysis (see for instance the important contributions of  Da Prato, Deboussche,  Nualart and others \cite{DaPrato1,DaPrato2,Gyongi,Gourcy,Kim,Manca}), interesting connections with geometry (see for instance the contributions by Cruzeiro and Malliavin \cite{Cruzeiro} or Davies, Truman, and Zhao \cite{Davies}), connections with the theory of superprocesses in \cite{Bonnet}  etc. This theoretical work was inspired by issues related to turbulence (see e.g. \cite{Albeverio, She, Sinai}) but also led to many new exciting applications (e.g. the work of Kiefer  \cite{Kiefer} that connects the random Burgers equation with polymer models, see also the review paper of Bec and Khanin \cite{Bec} on Burgers turbulence  and references therein  for details on other possible applications).

Another interesting problem related to the Burgers equation, both for the deterministic as well as for the stochastic version,  is that of  its optimal control  or controllability. This is an important problem both from the point of view of theoretical considerations as from the point of view of applications. Temam and coworkers \cite{Choi} studied this problem from the viewpoint of application in turbulence control and gave some positive answers to the question of whether a feedback control law can be a feasible way to drive the fluid flow to a desired final state. Such questions were further elaborated upon by Da Prato and Debussche \cite{DaPrato3} with the use of dynamic programming techniques. However, there has been a renewed interest on the  issue of controllability for  Burgers equation; see for instance \cite{Chapouly, Horsin, Glass} for recent related work.
These problems bring up the following two interesting directions.

The first one is related to the \emph{well-posedeness} of the forward stochastic Burgers equation, complemented with the feedback control law predicted by the dynamic programming equation. Following \cite{DaPrato3} such a closed loop equation may take the form
\begin{eqnarray}
dU = \left(\frac{1}{2}\,U_{xx} - U \, U_{x}\right) \, dt + \Phi(t, U) dt + Q\, dW(t),
\label{eqn:DAPRATO}
\end{eqnarray}
where $\Phi(t,U)$ is a functional specifying the feedback control law. The exact form of $\Phi$ depends on the cost criterion which is to be optimized, and at this point is irrelevant to our discussion. What is important, though, is that it introduces the need to study more complicated forms of the forward stochastic Burgers equation, that includes a potential like term $\Phi(t,U)$. This term may well be a random field itself.

The second one is related to the issue of \emph{controllability}. When one considers the problem of whether it is feasible to find a control procedure $W(t,U)$ which drives the system to a desired final state $\xi$ at a given time $T$, then it is appropriate to look at the original system as a final value problem with $U(T,x)=\xi$, rather than as an initial value problem. For instance, if one wishes to find $W$ so as to drive the system

\begin{eqnarray}
U_{t}-U \, U_{x} = -\frac{1}{2}\, U_{xx} + W(t, U)
\nonumber
\end{eqnarray}
to the desired state $U(T,x)=\xi$ then one should study the well-posedeness of the resulting backward equation
in the sense of the existence of a random field $U$ and a mapping $W(t,U)$ that satisfy the above problem. There are close connections between the problem of controllability and the well-posedeness of this backward problem;
see for instance the discussion of the famous Hilbert Uniqueness Method (HUM) proposed by J. L. Lions \cite{Lions} and J. M. Coron \cite{Coron} for the control of linear and nonlinear spatially extended systems, respectively.
This powerful idea requires a major revision so that it may be applicable to SPDEs, on account of problems related to the adaptivity of the solution to the filtration generated by the noise process when time is reversed.
Such questions led to the development of the theory of backward SPDEs (BSPDEs) (see for instance the work by Ma and Yong in \cite{Ma-Yong98} and \cite{Ma-Yong99} that have recently found applications in \cite{Peng}, \cite{Buckdahn} or \cite{Englezos}
to ``pathwise" utility optimization problems emerging from mathematical finance.  Using this powerful theory we may consider the problem of controllability of the stochastic Burgers equation,
by addressing the problem of well-posedeness of a backward stochastic Burgers equation that has the form
\begin{eqnarray}
&&dU = \left( -\frac{1}{2}\,U_{xx} - U \, U_{x} + \Phi(U,Z)\right)\, dt + Z \,dW(t),
\nonumber \\
&&U(T,x)=\xi, \nonumber
\end{eqnarray}
where $\xi$ is a given random final condition, and now we need to look for the pair of random fields $(U,Z)$. As expected the unknown random field $Z$ is related to the control procedure needed to drive the system to the desired final state.

In all the above, the stochasticity was supposed to be an additive term, imposed in the system through the external driving force $F(t,x,\omega)$. However, the randomness may well be
inserted in the model in terms of \emph{random} coefficients. For instance, one may consider  a model which is of the form of the Burgers equation, but with the viscosity term $\nu$ being
a random field rather than a constant parameter. Such a situation may be a good model for fluid turbulence, in which the viscosity term is determined through the distribution of turbulent eddies
in the flow. Another case where a Burgers equation with random coefficients could arise, is when we consider a random potential. This may either arise as a physical model
(e.g. as a model for a conducting fluid in the presence of external random electromagnetic fields) or in terms of a controlled problem, in the spirit of (\ref{eqn:DAPRATO}). Other possible motivations
could arise in the context of economic theory, and in particular within the context of portfolio selection in market models with random coefficients.

It is therefore the aim of this paper to study \emph{forward} as well as \emph{backward} stochastic Burgers equations with \emph{random} coefficients. Such problems have not been studied extensively in the literature.  One of our main concerns is to study whether a \emph{generalization} of the Cole-Hopf transformation can be obtained, that allows us to transform the \emph{quasilinear} Burgers equation with random coefficients to a \emph{linear} heat equation.
If this holds, then we may find solutions via this approach for the Burgers equation possibly in closed form, a fact that will allow us to obtain interesting information on their properties. Such information will provide important input for the relevant models. On the other hand, even if the solution of the heat equation in analytic form  is no more possible, due to the presence of the random coefficients, the reduction to a linear equation will allow us to deduce interesting qualitative information on the solutions of the Burger equation, such as comparison principles, Feynman-Kac representations etc.

In the forward case, we employ the Cole-Hopf transformation to a forward version of stochastic Burgers equation with random coefficients, linearizing it to a stochastic heat equation that does not contain a stochastic integral term; that is, it can be solved \emph{pathwise} as a deterministic one. A special version of forward stochastic Burgers equation with constant coefficients was studied in \cite{Villarroel}.
The backward case is far more delicate and our method of approach is to associate the stochastic Burgers equation with a system of ordinary forward-backward stochastic differential equations (FBSDEs) with random coefficients, through the four step scheme introduced in \cite{Ma Protter Yong}; an exposition of the theory on FBSDEs may also be found in \cite{Ma-Yong98}. To do so, firstly we establish this relation both for the Burgers equation and for its corresponding heat equation in the case of deterministic coefficients, developing a \emph{probabilistic} approach to the Cole-Hopf transformation that allows us to reinterpret it as a \emph{point} transformation between the two associated FBSDE systems. Using then this approach, we \emph{generalize} this transformation to backward stochastic Burgers equations with random coefficients. This generalization is by no means trivial and reflects interesting features concerning the nature of the equation. Through this extended transformation, we find the most general version of backward stochastic Burgers equation that is linearizable and reducible to a stochastic heat equation, subject to additional constraint equations that now appear due to the randomness of the coefficients. Examples of completely solvable backward stochastic Burgers equations are also presented for particular families of random coefficients.

In both cases, forward or backward, the resulting stochastic heat equation is linear and much easier to handle. Thus from the linear system and the generalized Cole-Hopf transformation we construct solutions
to stochastic Burgers equations with random coefficients and obtain \emph{stochastic} Feynman-Kac type representations for them. Our results indicate that the backward stochastic Burgers equations, whose study has been neglected so far in the literature, have more intricate structure than their forward counterparts  and can lead to interesting applications.
In particular, the controllability of a backward stochastic Burgers system that drives it to a predetermined final state amounts to its solvability subject to a suitably selected initial state control. On the other hand,
in a tax regulated financial market with a money market and a stock, a small investor may price and hedge a contingent claim, whose future value depends on the volatility of the stock, by selecting his portfolio according
to the solution of an FBSDE system and, in turn, to the solution of the associated backward stochastic Burgers equation.

A summary of the paper is as follows. In Section \ref{sec:RBEMs} we introduce a forward and a backward general version of stochastic Burgers equation with random coefficients, and present the necessary mathematical preliminaries. Section \ref{sec:PACH} connects the backward deterministic Burgers equation  with a class of FBSDEs, via the four step scheme, developing a probabilistic approach to the celebrated Cole-Hopf transformation. In Sections \ref{sec:FRBE} and \ref{sec:BRBE} we cope with the linearization of the already introduced forward and backward stochastic Burgers equations with random coefficients, respectively. Stochastic Feynman-Kac formulae are established for their solutions in Section \ref{sec:SFKF}. Finally, Section \ref{sec:AFB} illustrates the obtained results with applications to controllability and mathematical finance.
\vskip1cm
\section{Random Burgers Equations: Introduction of Two Models and Mathematical Preliminaries}\label{sec:RBEMs}
The main goal of this section is two-fold. The first goal is  to introduce two versions of  stochastic Burgers equations
with \emph{random} coefficients, a forward and a backward one, and motivate them as generalizations of the deterministic Burgers equation in the presence of noise and randomness in the coefficients. The second goal is to set the necessary notation and functional framework, and also present briefly the necessary mathematical tools that will be used in the paper.
\vskip0.5cm
\subsection{Two versions of the random Burgers equation}
In contrast with the deterministic case where the forward and the backward problem in time can be  related with the use of the simple time inversion transformation\footnote{even though the parabolic nature of the equation may introduce certain difficulties related to the unique continuation of the solution for negative times} $t \rightarrow -t$,  in the stochastic case with or without random coefficients this is no longer true, on account of technical but vital for the nature of the problem difficulties related with the adaptedness of the solution to the filtration generated by the noise process. Therefore, in the stochastic case the forward and the backward problems have to be formulated and treated differently. In both problems, we consider a 1-dimensional Brownian motion $W(\cdot)$ on
some filtered probability space $\big(\Omega, \mathcal{F},P; \mathbb{F}\big)$
with $\mathbb{F}\triangleq\mathcal{F}(\cdot)$ being the natural filtration generated by $W(\cdot),$ augmented by all the $P$-null sets
in $\mathcal{F}.$ Generalizations to higher dimensional or even infinite dimensional noise are feasible but are not pursued in the present work.
In addition, $\sigma$, $a$, $g$, $b,$ $e,$ $s,$ $m,$ and $f$ are assumed to be square integrable, real-valued random fields defined on
$[0,T]\times\mathbb{R}\times\Omega$ for a terminal time $T>0,$
such that for fixed $x\in \mathbb{R}$ they are $\mathbb{F}$-progressively measurable.

We consider first the \emph{forward} version of the stochastic Burgers equation with initial condition given by a square integrable random field $p:\mathbb{R}\times\Omega\rightarrow\mathbb{R}$; that is,
\begin{align}\label{eqn:forwardSBURGERSred0}
dU&=\bigg[\frac{1}{2} \ \sigma^2(t,x)U_{xx}+a(t,x)UU_{x}+g(t,x)U\Psi(t,U)+b(t,x)U_{x}\nonumber\\
&\quad \ \ +e(t,x)U+s(t,x)\Psi_{x}(t,U)+m(t,x)\Psi(t,U)+f(t,x)\bigg]dt\\
&\quad+\Psi(t,U)dW(t), \qquad 0< t\leq T,\quad x\in \mathbb{R},\nonumber\\
U(0,x)&=p(x),\qquad x\in \mathbb{R}.\nonumber\\ \nonumber
 \end{align}

Here $\Psi(t,U)$ is a \emph{known} function (or even random field) of $U$, pre-described by the form of the model. It reflects the effect of the fluctuations on the evolution of $U$, i.e., how the effect of $W(\cdot)$ transfers to an effect on $U$. In the simple case where $
\Psi(t,U)$ is a constant, we obtain a stochastic Burgers equation with additive noise.

The motivation for equation (\ref{eqn:forwardSBURGERSred0}) is rather straightforward. One may consider any physical situation where the evolution of the state of the system in question  is modeled in terms of Burgers equation. Then assume that this system is subject to external sources, whose behavior is random and subject to fluctuations (this accounts for the Wiener process term in the above system), while at the same time the coefficients of the model are also subject to uncertainties. This uncertainty is introduced by the random coefficients, in a sense that the coefficients of the model are functions, rather than constants, of a random process that evolves in parallel with the state of the system. This random process depends on the external noise term, therefore we assume that the random coefficients are functions which are adapted to the filtration generated by the Wiener process.

To illustrate the above points let us consider the example of turbulence modeling: the Wiener process models a body force for the fluid which presents fluctuations around an average body force.
 The random coefficients would correspond to a randomly changing viscosity term in the fluid. This is a very reasonable modeling assumption since for example the effective viscosity of the fluid depends on
 the eddie's formed in the turbulent flow which in turn depend on the random fluid velocity, which in turns depends on the external random body force. The above discussion can be easily transferred to other models.
 In conclusion we ought to comment on the formulation of a forward equation. We assume that the state of the system is known at $t=0$, and our aim through this model is to predict the possible future states
 of the system and provide information on their statistical properties.
This problem will be a useful model for a number of physical situations where the Burgers equation arises.

We now introduce the \emph{backward} problem. Here we assume that we know the final state of the system, which is the square integrable $\mathcal{F}_{T}$-measurable random field $p:\mathbb{R}\times\Omega\rightarrow\mathbb{R}$,
 and we wish to find which initial condition must be chosen in order to drive the system at time $T$ to this state.
The proper formulation of this problem is in terms of the BSPDE:

\begin{align}\label{eqn:SBURGERS0}
dU&=\bigg[-\frac{1}{2} \ \sigma^2(t,x)U_{xx}+a(t,x)UU_{x}+g(t,x)U\Psi^{U}+b(t,x)U_{x}\nonumber\\
& \qquad +e(t,x)U+s(t,x)\Psi_{x}^{U}+m(t,x)\Psi^{U}+f(t,x)\bigg]dt\\
&\quad+\Psi^{U}dW(t), \qquad  0\leq t < T,\quad x\in \mathbb{R}, \nonumber\\
U(T,x)&=p(x),\qquad x\in \mathbb{R}. \nonumber
\end{align}

It should be stressed here that now the unknowns are the pair of random fields $(U,\Psi^{U})$ and not just $U$. Therefore in contrast to what holds for the forward problem (\ref{eqn:forwardSBURGERSred0}), $\Psi(t,U)$ can no longer be thought of as given data to the problem, i.e. as pre-described by the model, but now it has to be specified as part of the solution of the problem. What is even more interesting is that $\Psi^{U}$ is \emph{uniquely} determined by $U$, as a properly defined functional of $U$  and not necessarily as a function of $U$.
In some sense $\Psi^{U}$ has to be interpreted as the (unique) auxiliary process needed to drive the system to the desired random final state. Therefore, it may have the interpretation of a \emph{control procedure} which acts on the system so as to drive it to a desired state. Furthermore, the compensation of this auxiliary process, allows us to obtain a solution to the backward problem which is adapted to the filtration generated by the Wiener process.

The important motivation for this problem arises from its connection with control problems. By the discussion above, one observes immediately that the backward problem is intimately related to the problem of controllability of the stochastic Burgers equation. This problem can be stated as follows: given an initial state can we drive the system by the proper control procedure  to a desired final state which is random but with prescribed statistical properties? As we have stated the problem, it allows us to characterize this initial state that will drive us to the desired final state, and at the same time through $
\Psi^{U}$ characterize the control procedure needed in the accomplishment of this task. The connection is much deeper; in particular, through the generalization of the Pontryagin maximum principle in the context of SPDEs one may show that the dual  (adjoint) system associated with a large variety of optimal control problems will have a form closely related to (\ref{eqn:SBURGERS0}).
\vskip0.5cm
\subsection{Notation, functional setting, and some important preliminary results}\label{subsec:PRE}

Throughout this paper, denote by
$C^{0,k}\big([0,T]\times\mathbb{R}^{n}\big)$
for any integer $k\geq 0$ the set of real-valued functions
on $[0,T]\times\mathbb{R}^{n}$ that are continuous with respect to the time variable
and continuously
differentiable up to order $k$ with respect to the spatial variable;
define accordingly the set $C^{k}(\mathbb{R}^{n})$. In addition, for any $1\leq
p\leq\infty$, any Banach space $\mathbb{X}$ with norm
$\|\cdot\|_{\mathbb{X}}$, and any sub-$\sigma$-algebra
$\mathcal{G}\subseteq\mathcal{F}$, denote by
\begin{itemize}
\item $\mathbb{L}^{p}_{\mathcal{G}}(\Omega,\mathbb{X})$ the
set of all $\mathbb{X}$-valued, $\mathcal{G}$-measurable random
variables $X$ such that
$E\|X\|_{\mathbb{X}}^{p}<\infty\; ;$

\item $\mathbb{L}^{p}_{\mathbb{F}}(0,T;\mathbb{X})$ the set
of all $\mathbb{F}$-progressively measurable, $\mathbb{X}$-valued
processes $X:[0,T]\times\Omega\rightarrow\mathbb{X}$ such
that $\int_{0}^{T}\|X(t)\|_{\mathbb{X}}^{p}dt<\infty, \
\text{a.s.} ;$

\item $\mathbb{L}^{p}_{\mathbb{F}}\big(0,T;\mathbb{L}^{p}(\Omega;\mathbb{X})\big)$ the set
of all $\mathbb{F}$-progressively measurable, $\mathbb{X}$-valued
processes $X:[0,T]\times\Omega\rightarrow\mathbb{X}$ such
that $\int_{0}^{T}E\|X(t)\|_{\mathbb{X}}^{p}dt<\infty\; ;$

\item $C_{\mathbb{F}}([0,T];\mathbb{X})$
the set of all continuous, $\mathbb{F}$-adapted processes
$X(\cdot,\omega):[0,T]\rightarrow\mathbb{X}$ for $P$-a.e.
$\omega\in\Omega$.
\end{itemize}
Moreover, define similarly the set
$C_{\mathbb{F}}\big([0,T];\mathbb{L}^{p}(\Omega;\mathbb{X})\big)$, and let
$\mathbb{R}^{+}$ stand for the positive real numbers.

By default in the sequel, all (in)equalities between random quantities are to be understood $dP$-almost surely, $dP\otimes dt$-almost everywhere or
$dP\otimes dt\otimes dx$-almost everywhere, as suitable in the situation at hand.
Furthermore, for every random field $\mathrm{F}:[0,T]\times\mathbb{R}^{n}\times\Omega\rightarrow\mathbb{R}$ denote by $A^{\mathrm{F}}, \Psi^{\mathrm{F}}:[0,T]\times\mathbb{R}^{n}\times\Omega\rightarrow\mathbb{R}$ the random fields for which $\mathrm{F}$ obtains the semimartingale decomposition
\[ \mathrm{F}(t,x)=\mathrm{F}(0,x)+\int_{0}^{t}A^{\mathrm{F}}(s,x)ds+\int_{0}^{t}\Psi^{\mathrm{F}}(s,x)dW(s)
 \ \ \ \forall\ \, (t,x) \in [0,T]\times\mathbb{R}^{n}. \]
Assume a similar notation for stochastic processes as well.

In order to carry
out computations regarding the change-of-variable formula for random fields, we recall the following useful
implication of the generalized It\^{o}-Kunita-Wentzell (IKW) formula (e.g.
\cite{Kunita}, Section 3.3, pp 92-93).

\begin{proposition}\label{prop:IKW}
Suppose that the random field
$\, \mathrm{F}:[0,T]\times\mathbb{R}^{n}\times\Omega\rightarrow\mathbb{R}$
is of class $C^{0,2}([0,T]\times\mathbb{R}^{n})$ and $A^{\mathrm{F}}, \Psi^{\mathrm{F}}:[0,T]\times\mathbb{R}^{n}\times\Omega\rightarrow\mathbb{R}$ are  $\mathbb{F}$-adapted random fields of class  $C^{0,1}([0,T]\times\mathbb{R}^{n})$ and $C^{0,2}([0,T]\times\mathbb{R}^{n})$, respectively. Furthermore,
let $\, X=(X^{(1)}, \ldots, X^{(n)})^{\ast}\, $ be a
vector of continuous semimartingales, 
where $\ A^{X^{(i)}}(\cdot)$ is an
almost surely integrable process and
$\Psi^{X^{(i)}}(\cdot)$ is
an $\mathbb{F}$-progressively measurable, almost surely square
integrable process. Then
$\mathrm{F}(\cdot,X(\cdot))$ is also a continuous
semimartingale, with decomposition
\begin{align*}
\mathrm{F}\big{(}t,X(t)\big{)}= &\mathrm{F}\big{(}0,X(0)\big{)}
 + \int_{0}^{t}A^{\mathrm{F}}\big{(}s,X(s)\big{)}ds + \int_{0}^{t}\Psi^{\mathrm{F}}\big{(}s,X(s)\big{)}dW(s)
 \nonumber\\
 &+ \sum_{i=1}^{n}\int_{0}^{t}\frac{\partial}{\partial x_{i}}\mathrm{F}\big{(}s,X(s)\big{)}A^{X^{(i)}}(s)ds \nonumber\\
 \end{align*}
 \begin{align}
\phantom{\mathrm{F}\big{(}t,X(t)\big{)}}
 &+\sum_{i=1}^{n}\int_{0}^{t}\frac{\partial}{\partial x_{i}}\mathrm{F}\big{(}s,X(s)\big{)}\Psi^{X^{(i)}}(s)dW(s)
\nonumber\\
\label{2b} & +
\sum_{i=1}^{n}\int_{0}^{t}\frac{\partial}{\partial x_{i}}\Psi^{\mathrm{F}}\big{(}s,X(s)\big{)}\Psi^{X^{(i)}}(s)ds\\
& + \frac{1}{2}
\sum_{i=1}^{n}\sum_{j=1}^{n}\int_{0}^{t}\frac{\partial^{2}}{\partial x_{i}\partial x_{j}}\mathrm{F}\big{(}s,X(s)\big{)}
\Psi^{X^{(i)}}(s)\Psi^{X^{(j)}}(s)ds \nonumber
\end{align}
 for every $0\leq t\leq T$.
\end{proposition}
\vskip 1cm
\section{Connection of Backward Burgers Equation with FBSDEs and a Probabilistic Approach to the Cole-Hopf Transformation}\label{sec:PACH}

For a given function $p:\mathbb{R}\rightarrow \mathbb{R}$ and $\sigma\neq 0$, consider the backward \emph{deterministic} Burger's equation of the form
\begin{equation}\label{eqn:BURGERS}
 \begin{split}
 U_{t}+ \frac{1}{2} \ \sigma^{2} \; U_{xx}-\sigma^2 U \, U_{x}&=0, \quad 0<t\leq T,\quad x\in\mathbb{R},\\
 U(T,x)&=p(x), \quad \ x\in\mathbb{R}.
\end{split}
\end{equation}
It is well known that using a transformation to a new variable $V:[0,T]\times\mathbb{R}\rightarrow\mathbb{R}^{+}$
such that $U=-\frac{\partial  }{\partial x} \ln V$, equation
(\ref{eqn:BURGERS}) is linearized assuming the form of the heat
equation
\begin{equation}\label{eqn:HEAT}
\begin{split}
 V_{t} + \frac{1}{2} \ \sigma^2 V_{xx}&=0, \quad 0<t\leq T,\quad x\in\mathbb{R},\\
  V(T,x)&= e^{-\int p(x)dx}\triangleq q(x), \quad \ x\in\mathbb{R}.
\end{split}
\end{equation}
This is the celebrated Cole-Hopf transformation, through which one may construct solutions of Burgers equation using appropriate solutions of the heat equation.

Let us revisit this linearization from a
probabilistic viewpoint. For any $x\in\mathbb{R}$, consider
the system of FBSDEs
\begin{align}\label{eqn:FBSDEB}
dX(t)&=\sigma \, dW(t), \quad  0<t\leq T, \nonumber\\
dY(t)&=\sigma Y(t)\,Z(t) \, dt + Z(t) \, dW(t), \quad 0\leq t<T,\\
X(0)&=x,\qquad Y(T)=p\big(X(T)\big),\nonumber
\end{align}
where $X(\cdot)$ is the forward process, $Y(\cdot)$ is the backward process,
and $Z(\cdot)$ is the auxiliary process needed for the well-posedness of
the problem, each defined on $[0,T]\times \Omega$. If one looks for a Markovian solution to this problem
of the form $Y(t)=U(t,X(t))$ then a simple application of It\^o's
formula
yields both that the deterministic function $U$ should satisfy equation
(\ref{eqn:BURGERS}) and the relationship $Z(t)=\sigma U_{x}(t,X(t))$.
Consider further the FBSDE
\begin{align}\label{eqn:FBSDEH}
dx(t)&=\sigma \, dW(t), \quad  0<t\leq T,\nonumber\\
dy(t)&= z(t) \, dW(t), \quad 0\leq t<T,\\
x(0)&=x,\qquad y(T)=q\big(x(T)\big)\nonumber
\end{align}
for the $\mathbb{F}$-adapted processes $x,y,z:[0,T]\times\Omega\rightarrow \mathbb{R},$ and follow the same steps as above to verify that for a Markovian solution
of the form
$y(t)=V(t,x(t))$  the deterministic function $V$ should satisfy equation
(\ref{eqn:HEAT}) and $z(t)=\sigma V_{x}(t,x(t))$.

The following issue is addressed. Taking into account that the PDEs
(\ref{eqn:BURGERS}) and (\ref{eqn:HEAT}) are related via the Cole-Hopf transformation,
we should expect that the probabilistic systems (\ref{eqn:FBSDEB}) and (\ref{eqn:FBSDEH}),
giving rise to these equations respectively, must be
related as well. Conversely, if
we find a point transformation between the variables $(X,Y,Z)$   and
$(x,y,z)$ then we could use it to reproduce the
Cole-Hopf transformation that allows us to get from the Burgers
equation to the heat equation.
In particular, we consider the point transformation between the solutions of (\ref{eqn:FBSDEB}) and (\ref{eqn:FBSDEH}):
\begin{align}
X&=x,
\nonumber \\
Y&=\mathcal{Y}(x,y,z),
\label{eqn:TRANSF} \\
Z&=\mathcal{Z}(x,y,z), \nonumber
\end{align}
where $\mathcal{Y}$,$\mathcal{Z}:\mathbb{R}^3\rightarrow\mathbb{R}$ are deterministic functions to be determined.
Then we have the following result.

\begin{proposition}\label{pro:PDEY}
Suppose that the process triplet $(x,y,z)$ solves (\ref{eqn:FBSDEH}), and the function $\mathcal{Y}$ belongs to $\mathcal{C}^{2,2,2}(\mathbb{R}^3)$ and
satisfies the PDE
\begin{align}\label{eqn:PDE}
 \frac{1}{2}\;\sigma^2 \, \mathcal{Y}_{xx}+\frac{1}{2}\;z^2 \,
\mathcal{Y}_{yy} + \frac{1}{2}\; h^2 \, \mathcal{Y}_{zz}
&+ \sigma \, z \, \mathcal{Y}_{xy} + h\, z\, \mathcal{Y}_{yz}+  \sigma \,h\, \mathcal{Y}_{xz}\\
&-\sigma^2 \, \mathcal{Y} \, \mathcal{Y}_{x} -  \sigma \,z\; \mathcal{Y} \, \mathcal{Y}_{y} - \sigma \,h\; \mathcal{Y}
\,\mathcal{Y}_{z}=0 \nonumber
\end{align}
for every $h\in\mathbb{R}.$ Furthermore let
\begin{align}\label{eqn:PDEZ}
\mathcal{Z}\big(x(t),y(t),z(t)\big)&=\sigma\;\mathcal{Y}_{x}\big(x(t),y(t),z(t)\big)+z(t)\;\mathcal{Y}_{y}
\big(x(t),y(t),z(t)\big)\\
& \quad +h(t)\mathcal{Y}_{z}\big(x(t),y(t),z(t)\big),\nonumber
\end{align}
in terms of the process $h:[0,T]\times \Omega\rightarrow \mathbb{R}$ such that $z(\cdot)$ has the local martingale representation $dz(t)= h(t) \, dW(t);$   cf. Remark \ref{rem:Z}.
Then the process triplet $(X,Y,Z)$ defined through the point transformation of (\ref{eqn:TRANSF}) solves (\ref{eqn:FBSDEB}).
\end{proposition}
\begin{proof}
Take the It\^o differential of $Y(t)=\mathcal{Y}\big(x(t),y(t),z(t)\big)$ to get
\begin{equation}\label{eqn:DIFY}
dY(t)= I\big(x(t),y(t),z(t);h(t)\big) \, dt + K\big(x(t),y(t),z(t);h(t)\big) \, dW(t),
\end{equation}
where
\begin{align}
I(x,y,z;h)&\triangleq \frac{1}{2}\; \sigma^2 \mathcal{Y}_{xx}(x,y,z) + \frac{1}{2}\;z^2 \,
\mathcal{Y}_{yy}(x,y,z) + \frac{1}{2}\;h^2 \mathcal{Y}_{zz}(x,y,z)\nonumber\\
& \ \ \  + \sigma z \mathcal{Y}_{xy}(x,y,z) + h z \mathcal{Y}_{zy}(x,y,z)+\sigma h \mathcal{Y}_{xz}(x,y,z),\, \nonumber
\\
K(x,y,z;h)&\triangleq \sigma \, \mathcal{Y}_{x} (x,y,z) + z \, \mathcal{Y}_{y}(x,y,z) + h \, \mathcal{Y}_{z}(x,y,z).
\nonumber
\end{align}
Thanks to (\ref{eqn:PDE}) and (\ref{eqn:PDEZ}), these functions satisfy the relationships
\begin{align}
I\big(x(t),y(t),z(t);h(t)\big)&=\sigma \mathcal{Y}\big(x(t),y(t),z(t)\big)\, \mathcal{Z}\big(x(t),y(t),z(t)\big),
\nonumber \\
K\big(x(t),y(t),z(t);h(t)\big)&=\mathcal{Z}\big(x(t),y(t),z(t)\big), \nonumber
\end{align}
and substituting them to (\ref{eqn:DIFY}) we obtain the stated result.
\end{proof}

\begin{remark}\label{rem:Z}
Note that applying It\^{o}'s differential to $z(t)=\sigma V_{x}\big(t,x(t)\big)$, in conjunction with (\ref{eqn:HEAT}),
allows us to take $dz(t)=h(t) \, dW(t)$ in the previous proposition.
\end{remark}

\begin{remark}\label{rem:DRCH}
One can easily verify that a solution of equation (\ref{eqn:PDE}) for any $h\in \mathbb{R}$ is the rational
function $\mathcal{Y}(x,y,z)=-z/\sigma y$. This interpreted in terms of
the solutions $U$ and $V$ of (\ref{eqn:BURGERS}) and (\ref{eqn:HEAT}), respectively, corresponds to the Cole-Hopf transformation, since
\[
U\big(t,x(t)\big)=Y(t)=\mathcal{Y}\big(x(t),y(t),z(t)\big)=-\frac{z(t)}{\sigma y(t)}=-\frac{V_{x}\big(t,x(t)\big)}{V\big(t,x(t)\big)}\,.
 \]
Hence, making use of probabilistic tools we managed to derive an alternative representation of this transformation in terms of a solution
of the PDE (\ref{eqn:PDE}).
\end{remark}
\vskip 1cm

 \section{Linearization of the Forward Random Burgers Equation}\label{sec:FRBE}

 In this section we consider the \emph{forward} version of the stochastic Burgers equation (\ref{eqn:forwardSBURGERSred0}) using the notation $\Psi(t,U)=\Psi^{U}$ and stress the fact that here $\Psi^{U}$ is a pre-described function of $U$.
 It is the object of the subsequent analysis to find the general form of the random coefficients of the above equation as well as of the function $\Psi^{U}$ such that
 this forward SPDE (FSPDE) can be transformed to a linear heat equation through the use of the Cole-Hopf transformation. Therefore, if we obtain somehow positive solutions for the linear one we may obtain solutions for the original Burgers equation with random coefficients by means of the Cole-Hopf transformation.

To this end, consider  the Cole-Hopf transformation $U=-\frac{\partial}{\partial x}\ln V$, in terms of a \emph{strictly positive} random field
$V\in C_{\mathbb{F}}\big([0,T];\mathbb{L}^2(\Omega;C^3(\mathbb{R};\mathbb{R}^{+}))\big)$
with a semimartingale decomposition such that $\Psi^{V}\in \mathbb{L}_{\mathbb{F}}^2\big(0,T;C^3(\mathbb{R})\big).$ For every $0< t\leq T,$ It\^{o}'s formula and the uniqueness of the semimartingale decomposition imply that

\begin{equation}\label{eqn:TRANS}
dU=-\frac{\partial}{\partial x}\left[\frac{dV}{V}-\frac{(\Psi^{V})^2}{2V^2}dt\right] \qquad \text{and} \qquad
\Psi^{U}=-\frac{\partial}{\partial x}\left(\frac{\Psi^{V}}{V}\right),
\end{equation}
\vskip0.5cm
\noindent
and substituting back into (\ref{eqn:forwardSBURGERSred0}) we have

\begin{align}\label{eqn:GGGG}
-\frac{\partial}{\partial x}\left[\frac{dV}{V}-\frac{(\Psi^{V})^2}{2V^2}dt\right]&=\Bigg\{-\frac{1}{2} \ \sigma^2(t,x)
\frac{\partial^3}{\partial x^3}\ln V+a(t,x)\frac{\partial}{\partial x}\ln V\frac{\partial^2}{\partial x^2}\ln V\nonumber\\
&\qquad\,+g(t,x)\frac{\partial}{\partial x}\ln V \frac{\partial}{\partial x}\left(\frac{\Psi^{V}}{V}\right)-b(t,x)\frac{\partial^2}{\partial x^2}\ln V\nonumber\\
&\qquad\,-e(t,x)\frac{\partial}{\partial x}\ln V-s(t,x)\frac{\partial^2}{\partial x^2}\left(\frac{\Psi^{V}}{V}\right)\\
&\qquad\, -m(t,x)\frac{\partial}{\partial x}\left(\frac{\Psi^{V}}{V}\right)+f(t,x)\Bigg\}dt\nonumber\\
 &\quad\ -\frac{\partial}{\partial x}\left(\frac{\Psi^{V}}{V}\right)dW(t). \nonumber
\end{align}
Inspired by the
 methodology implied by the Cole-Hopf transformation in the linearization of (\ref{eqn:introBURGERS}),  we accordingly integrate the last equation with respect to $x$ and assume that
$\sigma, a, g,$ and $s$ are independent of the spatial variable, i.e.,
\begin{equation}\label{def:COEF}
\sigma(t,x)\equiv \sigma(t), \quad a(t,x)\equiv a(t), \quad g(t,x)\equiv g(t), \quad \text{and} \quad s(t,x)\equiv s(t)
\end{equation}
are stochastic processes. Therefore, we arrive at
\begin{align}\label{eqn:BSPDEBEFLIN}
dV=\Bigg\{&\frac{1}{2} \ \sigma^2(t)V_{xx}-\frac{1}{2} \ \Big(\sigma(t)^2+a(t)\Big)\frac{V_{x}^{2}}{V}-g(t)V\int\frac{V_{x}}{V}\frac{\partial}{\partial x}\left(\frac{\Psi^{V}}{V}\right)dx
\qquad\qquad\qquad\qquad\nonumber\\
&+V \bigg[\int b(t,x)\frac{\partial^2}{\partial x^2}\ln V dx+\int e(t,x)\frac{V_{x}}{V}\; dx\nonumber\\
&\qquad\qquad\qquad\qquad\qquad\quad\,\, +\int m(t,x)\frac{\partial}{\partial x}\left(\frac{\Psi^{V}}{V}\right)dx\bigg]\\
&+s(t)\Psi_{x}^{V}-s(t)\frac{\Psi^{V}V_{x}}{V}-V\int f(t,x)dx+\frac{(\Psi^{V})^2}{2V}\Bigg\}dt+\Psi^{V}dW(t)\, \nonumber.
\end{align}
This is an SPDE for the random field $V$ where $\Psi^{V}$ is a function of $V$, the exact form of which determines the choice of the function $\Psi^{U}$ through the second equation of
(\ref{eqn:TRANS}). This  is in general a nonlinear SPDE, whose form can be as complicated or even worse than the original Burgers equation. However, as we shall show, upon proper selection of the coefficients and the function $\Psi^{V}$ (equivalently $\Psi^{U}$) this equation may reduce to a linear one.

In order to compute the integrals of (\ref{eqn:BSPDEBEFLIN}) that contain the function $\Psi^{V}$,
it suffices to consider
a special ansatz;
we assume that
\begin{equation}\label{def:MPH}
\Psi^{V}(t,x)\triangleq \ell(t)V_{x}(t,x), \qquad \forall \;\; (t,x)\in[0,T]\times\mathbb{R}
\end{equation}
for a stochastic process $\ell: [0,T]\times\Omega\rightarrow \mathbb{R}$ to be selected afterwards. This choice leads to a simple form for the function $\Psi^{U}$;
making  use of the Cole-Hopf transformation, the second equation of (\ref{eqn:TRANS}) gives
\begin{equation}\label{eqn:MPB}
\Psi^{U}(t,x)=\ell(t)U_{x}(t,x), \qquad \forall \;\; (t,x)\in[0,T]\times\mathbb{R}.
\end{equation}
We next employ equation (\ref{def:MPH}) to reformulate the integral expressions of (\ref{eqn:BSPDEBEFLIN}) as
$$
\int \frac{V_{x}}{V}\frac{\partial}{\partial x}\left(\frac{\Psi^{V}}{V}\right)dx
=\frac{\ell(t)}{2}\left(\frac{V_{x}}{V}\right)^2$$
and, in conjunction with integration-by-parts, as
\begin{align*}
&\ \int b(t,x)\frac{\partial^2}{\partial x^2}\ln V dx
+\int e(t,x)\frac{V_{x}}{V} \;dx+\int m(t,x)\frac{\partial}{\partial x}\left(\frac{\Psi^{V}}{V}\right)dx\\ \\
=&\Big[b(t,x)+\ell(t)m(t,x)\Big]\frac{V_{x}}{V}+\int \Big[e(t,x)-b_{x}(t,x)-\ell(t)m_{x}(t,x)\Big] \frac{V_{x}}{V}\;dx.
\end{align*}
Hence, FSPDE (\ref{eqn:BSPDEBEFLIN}) becomes
\begin{align*}
dV=\Bigg\{&\frac{1}{2}\Big(\sigma^2(t)+2s(t)\ell(t)\Big)V_{xx}+ \Big[b(t,x)+\ell(t)m(t,x)\Big]V_{x}\\
&-\frac{1}{2}\; \Big[\sigma^2(t)+a(t)+\big(g(t)+2s(t)\big)\ell(t)-\ell^2(t)\Big]\frac{V_{x}^2}{V}\\
& +V\int \Big[e(t,x)-b_{x}(t,x)-\ell(t)m_{x}(t,x)\Big] \frac{V_{x}}{V}\;dx\\
&   -V\int f(t,x)dx\Bigg\}dt+\ell(t)V_{x}dW(t)\, .
\end{align*}
For this equation to be a well-possed linear parabolic SPDE with respect to the unknown field $V$ (by analogy to the deterministic case where it reduced to  the heat equation (\ref{eqn:introHEAT})), the coefficient in front of the parabolic term  $V_{xx}$ needs to be
strictly positive, i.e.
\begin{equation}\label{ineq:coeffVxx}
\sigma^2(t)+2s(t)\ell(t)> 0, \quad \forall\ t\in[0,T],
\end{equation}
and the coefficients of the terms  $V_{x}^{2}/V$ and $V_{x}/V$  should vanish.
Regarding the  coefficient of the term $V_{x}^{2}/V$, this is the case by defining
 \begin{equation}\label{def:genell}
 \ell(t)\triangleq \frac{g(t)+2s(t)\pm \sqrt{D(t)}}{2}, \quad 0\leq t\leq T,
 \end{equation}
 where we have assumed that
 \begin{equation}\label{def:D}
 D(t)\triangleq \big[g(t)+2s(t)\big]^2+4\big[\sigma(t)^2+a(t)\big]\geq 0, \quad \forall \ t\in[0,T].
 \end{equation}
We may further eliminate the  coefficient of the term $V_{x}/V$ by setting
\begin{equation*}\label{def:e}
e(t,x)\triangleq k_{\, x}(t,x),
\quad \forall \;\; (t,x)\in[0,T]\times\mathbb{R},
\end{equation*}
where
\begin{equation*}\label{def:k}
k(t,x)\triangleq b(t,x)+\ell(t)m(t,x), \quad \forall \;\; (t,x)\in[0,T]\times\mathbb{R},
\end{equation*}
is the coefficient of the term $U_{x}$ in (\ref{eqn:forwardSBURGERSred0}) thanks to (\ref{eqn:MPB}).
Finally, since $b, \, m$ remain arbitrary random fields we may assume without loss of generality that $k$ is an arbitrary random field of same regularity as well. A summary of our findings is included in the subsequent result.

\begin{theorem}\label{thm:shdhsc}
Assuming  (\ref{ineq:coeffVxx})-(\ref{def:D}), if  $U\in C_{\mathbb{F}}\big([0,T];\mathbb{L}^2(\Omega;C^2(\mathbb{R}))\big)$ is
 a random field that solves FSPDE
\begin{align}\label{eqn:forwardSBURGERSredmore}
dU&=\Bigg[\frac{1}{2} \ \big(\sigma^2(t)+2s(t)\ell(t)\big)U_{xx}+\big(a(t)+g(t)\ell(t)\big)UU_{x}+k(t,x)U_{x}\nonumber\\
 &\quad \ \ +k_{\, x}(t,x)U
 +f(t,x)\Bigg]dt+\ell(t)U_{x}dW(t), \qquad 0< t\leq T,\quad x\in \mathbb{R},\\
U(0,x)&=p(x),\qquad x\in \mathbb{R}\nonumber
 \end{align}
 then the random field $V$, given by the Cole-Hopf transformation, is of class $C_{\mathbb{F}}\big([0,T];\mathbb{L}^2(\Omega;C^3(\mathbb{R};\mathbb{R}^{+}))\big)$ and solves FSPDE
\begin{equation}\label{eqn:BSPDESH}
\begin{split}
dV&=\left[\frac{1}{2}\;\big(
\sigma^2(t)+2s(t)\ell(t)\big)V_{xx}+  k(t,x)V_{x}-c(t,x)V \right]dt\\
 &\quad \, +\ell(t)V_{x}dW(t),\qquad 0< t\leq T,\qquad x\in \mathbb{R},\phantom{\Big\{}\\
V(0,x)&=e^{-\int p(x)dx}\triangleq q(x), \qquad x\in\mathbb{R},
\end{split}
\end{equation}
where
\begin{equation}\label{def:ce}
c(t,x)\triangleq \int f(t,x)dx+\bar{c}(t)
\end{equation}
for every $(t,x)\in [0,T]\times\mathbb{R}$ and some process $\bar{c}(\cdot)$.

Furthermore,  FSPDE (\ref{eqn:BSPDESH}) admits a solution of the form
\begin{equation}\label{def:SOLV}
V(t,x)\triangleq G\big(t,x+H(t)\big), \qquad 0\leq t\leq T,\qquad x\in \mathbb{R},
\end{equation}
where
\begin{equation}\label{def:SINTH}
H(t)\triangleq \int_{0}^{\,t}\ell(s)dW(s), \qquad 0\leq t\leq T;
\end{equation}
in particular,  $G$ is a positive random field  of class $C_{\mathbb{F}}\big([0,T];\mathbb{L}^2(\Omega;C^3(\mathbb{R};\mathbb{R}^{+}))\big)$ that solves FSPDE
\begin{align}\label{eqn:BSPDEG}
dG&=\bigg\{\frac{1}{2}\ \Big[\sigma^2(t)+2s(t)\ell(t)-\ell^{2}(t)\Big]G_{xx}(t,x)
+k\big(t,x-H(t)\big)G_{x}(t,x)\nonumber\\
&\quad \ \ -c\big(t,x-H(t)\big)G(t,x)\bigg\}dt \qquad\text{on}\quad (0,T]\times\mathbb{R},\\
G(0,x)&=q\big(x\big)\quad \text{on} \quad\mathbb{R},\nonumber
\end{align}
subject to the condition
\begin{equation}\label{ineq:coeffGxx}
\sigma^2(t)+2s(t)\ell(t)-\ell^{2}(t)>0,\qquad 0\leq t\leq T.
\end{equation}

\end{theorem}

\begin{proof}
The first statement of the theorem comes directly out of the preceding analysis. To prove the second assertion of the theorem,
we seek for positive solutions of (\ref{eqn:BSPDESH}) that have the form of (\ref{def:SOLV}).
Here, we assume that $G$ is a random field of class $C_{\mathbb{F}}\big([0,T];\mathbb{L}^2(\Omega;C^3(\mathbb{R};\mathbb{R}^{+}))\big)$
with a semimartingale decomposition
in terms of  the random field pair $(A^{G},\Psi^{G})\in C_{\mathbb{F}}\big([0,T];\mathbb{L}^2(\Omega;C^1(\mathbb{R}))\big)\times \mathbb{L}_{\mathbb{F}}^2\big(0,T;C^2(\mathbb{R})\big).$
Employing the generalized IKW formula to (\ref{def:SOLV}), in conjunction with
(\ref{def:SINTH}),
we obtain
\begin{align*}
dV=&\bigg[A^{G}\big(t,x+H(t)\big)+\frac{1}{2}\,\ell^{2}(t) G_{xx}\big(t,x+H(t)\big)+\ell(t)\Psi_{x}^{G}\big(t,x+H(t)\big)\bigg]dt\\
&+\Big[\ell(t)G_{x}\big(t,x+H(t)\big)+\Psi^{G}\big(t,x+H(t)\big)\Big]dW(t)\nonumber
\end{align*}
for every $(t,x)\in (0,T]\times \mathbb{R}.$ Making also use of (\ref{def:SOLV}), we rewrite (\ref{eqn:BSPDESH}) as

\begin{align*}
 \phantom{\Bigg\{_{A}}dV&=\bigg[\frac{1}{2} \,
 \big(\sigma^2(t)+2s(t)\ell(t)\big)G_{xx}\big(t,x+H(t)\big)+k(t,x)G_{x}\big(t,x+H(t)\big)\\
&\qquad-c(t,x)G\big(t,x+H(t)\big) \bigg]dt
+\ell(t)G_{x}\big(t,x+H(t)\big)dW(t), \nonumber
\end{align*}
for every $(t,x)\in (0,T]\times\mathbb{R}.$
A comparison between the last two equations leads to the definitions
\begin{align*}
 A^{G}(t,x)&\triangleq \frac{1}{2}\Big[\sigma^2(t)+2s(t)\ell(t)-\ell^{2}(t)\Big]G_{xx}(t,x)
+k\big(t,x-H(t)\big)G_{x}(t,x)\\
&\quad -c\big(t,x-H(t)\big)G(t,x), \\
 \Psi^{G}(\cdot,\cdot)&\triangleq 0.
\end{align*}
Finally, substituting them in
the semimartingale decomposition of $\, G\, $, we derive the FSPDE of (\ref{eqn:BSPDEG}),
 whose well-posedness is ensured by (\ref{ineq:coeffGxx}) and the initial condition comes from (\ref{eqn:BSPDESH}) and (\ref{def:SOLV}) for $t=0$.
Regarding the positivity of $G$ see the following remark.
\end{proof}

\begin{remark}
 FSPDE (\ref{eqn:BSPDEG}) is a linear PDE with random coefficients, but does not contain any stochastic integral. This means that it can be treated pathwise rather than in an It\^o integration sense, a fact that simplifies its analysis immensely, and allows us to use qualitative results, e.g. maximum principles, to study properties of its solutions like positivity, monotonicity etc. In fact, for the special case of
constant coefficients with $k(\cdot,\cdot)=c(\cdot,\cdot)=0$ this PDE reduces to  (\ref{eqn:introHEAT}).
\end{remark}

\begin{remark} It is not hard to see that all the above inequalities of (\ref{ineq:coeffVxx}), (\ref{def:D}) and (\ref{ineq:coeffGxx})
hold simultaneously for various choices of the processes of (\ref{def:COEF}); e.g., in terms of process $\sigma(\cdot)$ we may set
$ \, a(\cdot)\triangleq\sigma^2(\cdot), \ g(\cdot)\triangleq\sigma(\cdot),  \  s(\cdot)\triangleq-\sigma(\cdot),$ and then $\ell(\cdot)\triangleq-2\sigma(\cdot)$ according to (\ref{def:genell}).

\end{remark}
\begin{remark}
The methodology presented in this section for the linearization of the forward random Burgers equation,
driven by a single Brownian motion, may be also generalized for spatially dependent noise; that is
$$d\widetilde{W}_\kappa(t,x)\triangleq \sum_{i=1}^{\kappa} \phi_{i}(x)\, dW_i(t), \quad \forall \;\; (t,x)\in [0,\infty)\times\mathbb{R},$$
where $W_i(\cdot), \, i=1,2,...,\kappa,$ are standard, real-valued Wiener processes and $\phi_i,\, i=1,2,...,\kappa,$ are functions modeling the spatial autocorrelation of the noise through the covariance operator, e.g. they can be a basis for $\mathbb{L}^2{(\mathbb{R})}.$
Equations of this type find interesting applications in interface modeling, of which a typical example is the Kardar-Parisi-Zhang (KPZ) equation (cf. \cite{Kardar}).
In this case, It\^{o}'s lemma and a formal differentiation of the associated to our analysis forward random
KPZ equation

\begin{align*}
dR&=\Bigg[\frac{1}{2} \ \sigma^2(t)R_{xx}-\frac{1}{2} \ \sigma^2(t)R_{x}^{\,2}+k(t,x)R_{x}+c(t,x)\Bigg]dt\\
&\quad \ +d\widetilde{W}_\kappa(t,x), \quad 0< t\leq T,\ x\in \mathbb{R},\nonumber\\\\
R(0,x)&=\rho(x),\qquad x\in \mathbb{R}\nonumber
 \end{align*}

 \noindent
lead, by introducing the random field $U\triangleq R_x,$ by considering the random field $\rho$ of class $\mathbb{L}^2\big(\Omega; C^1(\mathbb{R})\big)$  and recalling (\ref{def:ce}), to  the forward stochastic Burgers equation

\begin{align*}
dU&=\Bigg[\frac{1}{2} \ \sigma^2(t)U_{xx}-\sigma^2(t)UU_{x}+k(t,x)U_{x}+k_x(t,x)U+f(t,x)\Bigg]dt\\
&\quad \ +\frac{\partial}{\partial x}d\widetilde{W}_\kappa(t,x), \quad 0< t\leq T,\ x\in \mathbb{R},\nonumber\\
U(0,x)&=\rho_x(x),\qquad x\in \mathbb{R}\nonumber;
\end{align*}
this is the analogue of FSPDE  (\ref{eqn:forwardSBURGERSred0}) with conservative noise, i.e., the term of
$dW(t)$ is replaced by $\frac{\partial}{\partial x}d\widetilde{W}_\kappa(t,x)$.
Then, the random field $V\triangleq e^{-R}$ solves the linear forward
stochastic heat equation
\begin{align*}
dV&=\Bigg[\frac{1}{2} \ \sigma^2(t)V_{xx}+k(t,x)V_{x}+\left(\frac{1}{2}-f(t,x)\right)V\Bigg]dt\\
&\quad \ -Vd\widetilde{W}_\kappa(t,x), \quad 0< t\leq T,\ x\in \mathbb{R},\nonumber\\
V(0,x)&=e^{-\rho(x)},\qquad x\in \mathbb{R}.\nonumber
 \end{align*}

 The above calculus can be taken to the limit as $d\rightarrow \infty,$ where the series of $\widetilde{W}_\kappa$ can either be interpreted as a $Q$-Wiener process or
 as a cylindrical Brownian motion (cf. \cite{DaPrato4}) depending on the choice of the functions $\phi_i, \, i=1,2,...,\kappa.$ Of course, this may further require modifications in It\^{o}'s lemma, while in case of the cylindrical Brownian motion this limit has to be interpreted in a weak sense.
 Note also that a slightly different version of the KPZ equation  has been studied in
\cite{Bertini} for the fluctuation field of the interface profile in a microscopic growth model subject to a Wick renormalization of the nonlinearity.
\end{remark}
\vskip1cm

\section{Linearization of the Backward Random Burgers Equation}\label{sec:BRBE}

In Section \ref{sec:PACH}  we showed how the backward deterministic Burgers equation can be linearized using its connection with a system of FBSDEs,
through which we obtained an alternative probabilistic derivation of the celebrated Cole-Hopf transformation. In the present section we address the question of whether this analysis
 can be generalized for the \emph{backward} stochastic Burgers equation with random coefficients of (\ref{eqn:SBURGERS0}).  It will turn out that in order to answer this problem it is no longer enough to use the Cole-Hopf transformation that works for the deterministic or the forward random problem but rather we have to extend it in such a fashion as to take into account the intricate nature of the backward problem, and the presence of the related unknown process pair $(U,\Psi^{U})$.

For the simplicity of our mathematical
program which aims to the linearization of (\ref{eqn:SBURGERS0}), we also assume that the coefficients are bounded and infinitely differentiable with respect to the spatial variable $x$, with all their partial
derivatives bounded as well.
Carrying out though our analysis, we shall see that both weaker regularity conditions in $x$ may be imposed
 and explicit relationships among the coefficients will be in need.
\vskip0.5cm
\subsection{Connection with FBSDEs of random coefficients}
We consider the
forward process $X(\cdot)$ that is the solution of $X(t)=x+\int_{0}^{t}\sigma(s,X(s))dW(s)$ for $(t,x)\in [0,T]\times \mathbb{R},$
and seek for the corresponding backward given by $Y(t)=U(t,X(t))$
in terms of a solution $(U,\Psi^{U})\in C_{\mathbb{F}}\big([0,T];\mathbb{L}^2(\Omega;C^3(\mathbb{R}))\big)\times \mathbb{L}_{\mathbb{F}}^2\big(0,T;C^2(\mathbb{R})\big)$ of BSPDE (\ref{eqn:SBURGERS0}).
Note here that we have considered an additional order of smoothness for the pair $(U,\Psi^{U})$
than is required to constitute a solution of (\ref{eqn:SBURGERS0}). Nevertheless, this extra assumption
allows the application of the generalized
IKW formula, which coupled with (\ref{eqn:SBURGERS0}) yields

\begin{align*}
dY(t)=&\bigg\{\Big[a\big(t,X(t)\big)U_{x}\big(t,X(t)\big)+g\big(t,X(t)\big)\Psi^{U}\big(t,X(t)\big)\Big]U\big(t,X(t)\big)\\
&\ \ +b\big(t,X(t)\big)U_{x}\big(t,X(t)\big)
+e\big(t,X(t)\big)U\big(t,X(t)\big)\\
&\ \ +\Big(s\big(t,X(t)\big)+\sigma\big(t,X(t)\big)\Big)\Psi_{x}^{U}\big(t,X(t)\big)+m\big(t,X(t)\big)\Psi^{U}\big(t,X(t)\big)\\
&\ \ +f\big(t,X(t)\big)\bigg\}dt+ \Big[\sigma\big(t,X(t)\big)U_{x}\big(t,X(t)\big)+\Psi^{U}\big(t,X(t)\big)\Big]dW(t),
\end{align*}
for  $0\leq t<T.$
In order both to be consistent with the form of the deterministic Burgers equation of (\ref{eqn:BURGERS})
and to avoid matters of technical fuss in our following analysis, we shall consider that $g(\cdot)=\sigma(\cdot)\neq 0$ is a stochastic process.
Additionally, by setting
\begin{equation}
a(t,x)\triangleq g(t)\sigma(t), \quad b(t,x)\triangleq \sigma(t) m(t,x), \quad \text{and} \quad s(t,x)\triangleq-\sigma(t)
\end{equation}
for $(t,x)\in[0,T]\times\mathbb{R},$ we obtain the first part of the following result.
\begin{proposition} \label{prop:BSPDEANDFBSDEU}
(i) Suppose that the pair of random fields  $(U,\Psi^{U})$ is of class $ C_{\mathbb{F}}\big([0,T];\mathbb{L}^2(\Omega;C^3(\mathbb{R}))\big)\times \mathbb{L}_{\mathbb{F}}^2\big(0,T;C^2(\mathbb{R})\big)$
and solves the stochastic Burgers type BSPDE
\begin{align}\label{eqn:SBURGERSredmor}
 dU&=\bigg[-\frac{1}{2}\;    \sigma^2(t)U_{xx}+\sigma^{2}(t)UU_{x}+\sigma(t)U\Psi^{U}+\sigma(t)m(t,x)U_{x}\nonumber\\
 &\quad \ \ \, +e(t,x)U-\sigma(t)\Psi_{x}^{U}+m(t,x)\Psi^{U}+f(t,x)\bigg]dt\\
 &\quad+\Psi^{U}dW(t),\quad 0\leq t<T,\quad x\in \mathbb{R},\nonumber\\
U(T,x)&=p(x),\qquad x\in \mathbb{R}.\nonumber
 \end{align}
 \vskip0.3cm
 Then the triplet $(X,Y,Z)$ of stochastic processes, given by
 \begin{equation}\label{def:triplBSB}
 \begin{split}
 X(t)&\triangleq x+\int_{0}^{t}\sigma(s)dW(s),\qquad
 Y(t)\triangleq U\big(t,X(t)\big),\\
 \text{and} \qquad \ Z(t)&\triangleq \sigma(t)U_{x}\big(t,X(t)\big)+\Psi^{U}\big(t,X(t)\big)
 \end{split}
 \end{equation}
 \vskip0.3cm
 \noindent
 for every $t\in[0,T],$ satisfies the FBSDEs
 
\begin{equation}\label{eqn:FBSDESBred}
  \begin{split}
  dX(t)&=\sigma(t)dW(t),\quad 0<t\leq T,\\
  dY(t)&=\Big[\sigma(t)Y(t)Z(t)+e\big(t,X(t)\big)Y(t)+m\big(t,X(t)\big) Z(t)\\
  &\quad\ \ +f\big(t,X(t)\big)\Big]dt+Z(t)dW(t),\quad 0\leq t<T,\\
  X(0)&=x, \qquad Y(T)=p\big(X(T)\big).
  \end{split}
  \end{equation}
\\
(ii) Consider that the pair of random fields $(V,\Psi^{V})$ belongs to the class $ C_{\mathbb{F}}\big([0,T];\mathbb{L}^2(\Omega;C^4(\mathbb{R}))\big)\times \mathbb{L}_{\mathbb{F}}^2\big(0,T;C^3(\mathbb{R})\big)$
and satisfies the stochastic heat type BSPDE

\begin{align}\label{eqn:IKWG}
dV&=\Big\{-\frac{1}{2}\,\sigma^2(t) V_{xx}+\sigma(t)d(t,x)V_{x}-\sigma(t)\Psi_{x}^{V}+d(t,x)\Psi^{V}\nonumber\\
&\quad\quad+c(t,x)V\Big\}dt+\Psi^{V}dW(t),\quad 0\leq t<T,\quad x\in \mathbb{R},\\
V(T,x)&=q(x),\qquad x\in \mathbb{R}.\nonumber
\end{align}
Then the triplet $(x,y,z)$ of stochastic processes, defined by

\begin{equation}\label{def:triplBSH}
 \begin{split}
 x(t)&\triangleq x+\int_{0}^{t}\sigma(s)dW(s),\qquad
 y(t)\triangleq V\big(t,x(t)\big),\\
 \text{and} \qquad z(t)&\triangleq \sigma(t)V_{x}\big(t,x(t)\big)+\Psi^{V}\big(t,x(t)\big)
 \end{split}
 \end{equation}
 \vskip0.3cm
 \noindent
 for every $t\in[0,T],$ solves the FBSDEs
 
\begin{equation}
  \begin{split}\label{eqn:FBSDESH}
  dx(t)&=\sigma(t)dW(t), \quad 0<t\leq T,\\
  dy(t)&=\Big[c\big(t,x(t)\big)y(t)+d\big(t,x(t)\big)z(t)\Big]dt+z(t)dW(t),\quad 0\leq t<T,\\
  x(0)&=x,\qquad y(T)=q\big(x(T)\big).
  \end{split}
  \end{equation}
  \end{proposition}
\vskip0.5cm
\begin{proof} We need only to show (ii). Indeed, the BSPDE (\ref{eqn:FBSDESH}) follows readily  through (\ref{eqn:IKWG}), (\ref{def:triplBSH}), and the generalized IKW formula.
\end{proof}
\smallskip

The first part of this proposition provides a solution to the FBSDEs (\ref{eqn:FBSDESBred}) by means of a solution of the Burgers BSPDE (\ref{eqn:SBURGERSredmor}).
In Section \ref{sec:PACH} though, we saw that the Cole-Hopf transformation
reduces the Burgers equation (\ref{eqn:BURGERS})
to the linear heat equation (\ref{eqn:HEAT}). Therefore, its next statement addresses the most general type of a backward linear stochastic heat equation (cf. (\ref{eqn:IKWG})) that is associated with an FBSDE system (cf. (\ref{eqn:FBSDESH})).
\vskip0.5cm

 \subsection{Characterization  of the generalized Cole-Hopf transformation  in terms of an SPDE}
  \vskip0.3cm
  Our goal in this subsection is to investigate wether there exists a \emph{stochastic} version of the Cole-Hopf that reduces (\ref{eqn:SBURGERSredmor}) to (\ref{eqn:IKWG}) for appropriate random fields $c,d:[0,T]\times\mathbb{R}\times\Omega\rightarrow\mathbb{R}$
   and $q:\mathbb{R}\times\Omega\rightarrow\mathbb{R}$.
 According to Remark \ref{rem:DRCH}, in the deterministic case the Cole-Hopf transformation is represented by the point transformation (\ref{eqn:TRANSF}).
 Thus, by analogy we shall seek
for a point transformation of  the form
\begin{align}
X(t)&=x(t),
\nonumber \\
Y(t)&=\mathcal{Y}\big(t,x(t),y(t),z(t)\big),
\label{eqn:STRANSF} \\
Z(t)&=\mathcal{Z}\big(t,x(t),y(t),z(t)\big),
\nonumber
\end{align}
that relates the solutions $(X,Y,Z)$ and $(x,y,z)$ of (\ref{eqn:FBSDESBred}) and (\ref{eqn:FBSDESH}) respectively,
in terms of suitable random fields $\mathcal{Y},$ $\mathcal{Z}:[0,T]\times\mathbb{R}^3\times\Omega\rightarrow\mathbb{R}$.
This transformation
will provide explicitly the Cole-Hopf
one in the generalized case of \emph{stochastic} coefficients.
\begin{assumption}\label{as:semisigmag}
Let the stochastic coefficient $\sigma(\cdot)$ obtain a
semimartingale decomposition,
where $A^{\sigma}(\cdot)$ is an integrable a.s., $\mathbb{F}$-progressively measurable stochastic process, and $\Psi^{\sigma}(\cdot)$ is a square-integrable a.s., $\mathbb{F}$-progressively measurable one. Furthermore, let the latter process obtain a similar semimartingale decomposition in terms of the stochastic processes $A^{\Psi^{\sigma}}(\cdot)$ and $\Psi^{\Psi^{\sigma}}(\cdot),$ respectively.
\end{assumption}

\begin{proposition}\label{prop:PRSCH}
 Consider Assumption \ref{as:semisigmag}. Let also $(x,y,z)$ be a solution of FBSDEs (\ref{eqn:FBSDESH}), $(\mathcal{Y},\Psi^{\mathcal{Y}})$ be a random field pair of class
 $C_{\mathbb{F}}\big([0,T];\mathbb{L}^2(\Omega;C^3(\mathbb{R}^3))\big)$ $\times \mathbb{L}_{\mathbb{F}}^2\big(0,T;C^2(\mathbb{R}^3)\big)$
 and satisfies the BSPDE

\begin{align}\label{eqn:BSPDEY}
d\mathcal{Y}=&\Bigg\{-\frac{1}{2} \ \sigma^2(t)\mathcal{Y}_{xx}-\frac{1}{2}\ z^2\mathcal{Y}_{yy}-\frac{1}{2}\ h^2\mathcal{Y}_{zz}-\sigma(t)z\mathcal{Y}_{xy}-\sigma(t)h\mathcal{Y}_{xz}-zh\mathcal{Y}_{yz}\nonumber\\ 
&\quad +\sigma^{2}(t)\mathcal{Y}\mathcal{Y}_{x}+\sigma(t)z\mathcal{Y}\mathcal{Y}_{y}
+\sigma(t)h\mathcal{Y}\mathcal{Y}_{z}+\sigma(t)m(t,x)\mathcal{Y}_{x}\nonumber\\ 
&\qquad\qquad\qquad\qquad\qquad\quad\,\,\,-\Big(d(t,x)z+c(t,x)y-m(t,x)z\Big)\mathcal{Y}_{y}\nonumber\\
\phantom{d\mathcal{Y}=}
&\ \ \,  -\Bigg[\bigg(A^{\sigma}(t)+\sigma^2(t)d_{x}\big(t,x(t)\big)-d\big(t,x(t)\big)\Psi^{\sigma}(t)
\\
&\qquad\qquad\qquad\qquad\qquad\qquad\,\,\quad +\sigma(t)c\big(t,x(t)\big)-\frac{(\Psi^{\sigma}(t))^2}{\sigma(t)}\bigg)\frac{z(t)}{\sigma(t)}\nonumber\\
&\qquad\ \ \,+\frac{1}{2}\;\sigma^2(t)\Psi_{xx}^{V}\big(t,x(t)\big)-\bigg(\sigma(t)d\big(t,x(t)\big)
+\Psi^{\sigma}(t)\bigg)\Psi_{x}^{V}\big(t,x(t)\big)\nonumber\\
&\qquad\ \ \, -\bigg(A^{\sigma}(t)-d\big(t,x(t)\big)\Psi^{\sigma}(t)\nonumber\\
&\qquad\qquad\qquad\quad\,
+\sigma(t)c\big(t,x(t)\big)-\frac{(\Psi^{\sigma}(t))^2}{\sigma(t)}\bigg)\frac{\Psi^{V}\big(t,x(t)\big)}{\sigma(t)}\nonumber\\
&\qquad \ \ \, +A^{\Psi^{{V}}}(t)+\sigma(t)\Psi_{x}^{\Psi^{V}}\big(t,x(t)\big)-\left(d\big(t,x(t)\big)+\frac{\Psi^{\sigma}(t)}{\sigma(t)}\right)\Psi^{\Psi^{V}}
\nonumber\\
&\qquad\ \ \, +\sigma(t)c_{x}\big(t,x(t)\big)y(t)+h\left(\frac{\Psi^{\sigma}(t)}{\sigma(t)}+d\big(t,x(t)\big)-m(t,x)\right)\Bigg]\mathcal{Y}_{z}\nonumber\\
&\quad +e(t,x)\mathcal{Y}+\sigma(t)\mathcal{Y}\Psi^{\mathcal{Y}}
-\sigma(t)\Psi^{\mathcal{Y}}_{x}-z\Psi^{\mathcal{Y}}_{y}-h\Psi^{\mathcal{Y}}_{z}+m(t,x)\Psi^{\mathcal{Y}}\nonumber\\
&\ \ \, +f(t,x)\Bigg\}dt+\Psi^{\mathcal{Y}}dW(t)\qquad \text{on}\qquad [0,T)\times\mathbb{R}^3\nonumber
\end{align}
for any $h\in\mathbb{R}$, and
\begin{equation}\label{eqn:Z}
\mathcal{Z}\triangleq \Psi^{\mathcal{Y}}+\sigma(t)\mathcal{Y}_{x}+z(t)\mathcal{Y}_{y}+h(t)\mathcal{Y}_{z}
\end{equation}
evaluated at $\big(t,x(t),y(t),z(t)\big)$ on $[0,T).$ Then the process triplet $(X,Y,Z)$
given by the point transformation (\ref{eqn:STRANSF}) solves the FBSDEs (\ref{eqn:FBSDESBred}).
\end{proposition}

\begin{proof}
Firstly we provide the semimartingale decomposition of the process $z(\cdot)$ in (\ref{def:triplBSH}). Given Assumption \ref{as:semisigmag}, considering that the pair  $(V,\Psi^{V})$ belongs to
 $ C_{\mathbb{F}}\big([0,T];\mathbb{L}^2(\Omega;C^4(\mathbb{R}))\big)\times \mathbb{L}_{\mathbb{F}}^2\big(0,T;C^3(\mathbb{R})\big)$, and assuming that the random field $\Psi^{V}$ has a semimartingale decomposition
 for $(A^{\Psi^{V}},\Psi^{\Psi^{V}})\in C_{\mathbb{F}}\big([0,T];\mathbb{L}^2(\Omega;C^1(\mathbb{R}))\big)\times \mathbb{L}_{\mathbb{F}}^2\big(0,T;C^2(\mathbb{R})\big),$
employ the product rule and the generalized IKW formula
to process $z(\cdot)$ of (\ref{def:triplBSH}), in conjunction with (\ref{eqn:IKWG}),
to get that
\begin{align*}
dz(t)=&\Bigg\{A^{\Psi^{V}}\big(t,x(t)\big)+\sigma(t)\bigg(\sigma(t)d\big(t,x(t)\big)+\Psi^{\sigma}(t)\bigg)V_{xx}\big(t,x(t)\big)\\
&\,\,\ +\bigg[A^{\sigma}(t)+\sigma^2(t)d_{x}\big(t,x(t)\big)+\sigma(t)c\big(t,x(t)\big)\bigg]V_{x}\big(t,x(t)\big)\\
 &\,\,\ +\sigma(t)c_{x}\big(t,x(t)\big)V\big(t,x(t)\big)
+\frac{1}{2}\;\sigma^2(t)\Psi_{xx}^{V}\big(t,x(t)\big)\\
&\,\,\ +\Big(\sigma(t)d\big(t,x(t)\big)+\Psi^{\sigma}(t)\Big)\Psi_{x}^{V}\big(t,x(t)\big)\\
&\,\,\ +\sigma(t)d_{x}\big(t,x(t)\big)\Psi^{V}\big(t,x(t)\big)+\sigma(t)\Psi_{x}^{\Psi^{V}}\big(t,x(t)\big)\Bigg\}dt\\
&+\bigg[\Psi^{\Psi^{V}}\big(t,x(t)\big)+\sigma^2(t)V_{xx}\big(t,x(t)\big)\\
&\qquad\qquad\qquad\quad\,+\Psi^{\sigma}(t)V_{x}\big(t,x(t)\big)
+2\sigma(t)\Psi_{x}^{V}\big(t,x(t)\big)\bigg]dW(t)
\end{align*}
for $0\leq t< T.$ Therefore, set
\begin{equation*}
 h(t)\triangleq \Psi^{\Psi^{V}}\big(t,x(t)\big)+\sigma^2(t)V_{xx}\big(t,x(t)\big)+\Psi^{\sigma}(t)V_{x}\big(t,x(t)\big)
+2\sigma(t)\Psi_{x}^{V}\big(t,x(t)\big)
\end{equation*}
for every $t\in[0,T]$ and recall the third equation of (\ref{def:triplBSH}) to derive the relationships
\begin{align}\label{eqn:DIFFV}
V_{x}\big(t,x(t)\big)&=\frac{z(t)-\Psi^{V}\big(t,x(t)\big)}{\sigma(t)}\ , \\
V_{xx}\big(t,x(t)\big)&=\frac{h(t)-\Psi^{\Psi^{V}}\big(t,x(t)\big)}{\sigma^2(t)}
+\frac{\Psi^{\sigma}(t)\bigg(\Psi^{V}\big(t,x(t)\big)-z(t)\bigg)}{\sigma^3(t)}\nonumber\\
\nonumber\\
&\quad\, -\frac{2\Psi_{x}^{V}\big(t,x(t)\big)}{\sigma(t)}\ , \nonumber
\end{align}
and deduce eventually for any $t\in[0,T)$ the semimartingale decomposition

\begin{align}\label{eqn:SMDz}
dz(t)=&\Bigg\{\bigg[A^{\sigma}(t)+\sigma^2(t)d_{x}\big(t,x(t)\big)-\Psi^{\sigma}(t)d\big(t,x(t)\big)\nonumber\\
&\qquad\qquad\qquad\qquad\qquad\,\,\,\,\quad+\sigma(t)c\big(t,x(t)\big)-\frac{(\Psi^{\sigma}(t))^2}{\sigma(t)}\bigg]\frac{z(t)}{\sigma(t)}\\
&\quad+\frac{1}{2}\;\sigma^2(t)\Psi_{xx}^{V}\big(t,x(t)\big)-\bigg(\sigma(t)d\big(t,x(t)\big)
+\Psi^{\sigma}(t)\bigg)\Psi_{x}^{V}\big(t,x(t)\big)\nonumber
\end{align}
\begin{align*}
\phantom{dz(t)=}
&\quad-\bigg[A^{\sigma}(t)-\Psi^{\sigma}(t)d\big(t,x(t)\big)\\
&\qquad\qquad\  \ \,\,\,+\sigma(t)c\big(t,x(t)\big)-\frac{(\Psi^{\sigma}(t))^2}{\sigma(t)}\bigg]\frac{\Psi^{V}\big(t,x(t)\big)}{\sigma(t)}\nonumber\\
&\quad +A^{\Psi^{{V}}}(t)+\sigma(t)\Psi_{x}^{\Psi^{V}}\big(t,x(t)\big)-\left(d\big(t,x(t)\big)+\frac{\Psi^{\sigma}(t)}{\sigma(t)}\right)\Psi^{\Psi^{V}}
\nonumber\\
&\quad+\sigma(t)c_{x}\big(t,x(t)\big)y(t)+\bigg(\frac{\Psi^{\sigma}(t)}{\sigma(t)}+d\big(t,x(t)\big)\bigg)h(t)\Bigg\}dt
+h(t)dW(t)\nonumber.
\end{align*}

We are ready now to argue that BSPDE (\ref{eqn:BSPDEY}) constitutes a sufficient condition
for the point transformation (\ref{eqn:STRANSF})  to produce solutions
of the FBSDEs (\ref{eqn:FBSDESBred}) directly out of solutions of the FBSDEs (\ref{eqn:FBSDESH}). Indeed, this follows directly
from the generalized IKW formula applied to the second equation of (\ref{eqn:STRANSF}), in combination with (\ref{eqn:FBSDESH}), (\ref{eqn:SMDz}), and (\ref{eqn:BSPDEY}).
\end{proof}

We shall focus next on establishing an equivalent characterization
for the solutions of BSPDE (\ref{eqn:BSPDEY}), in order to determine the point transformation of (\ref{eqn:STRANSF}). A byproduct of this analysis will be the Cole-Hopf transformation
between the solutions $U$ and $V$ of the BSPDEs (\ref{eqn:SBURGERSredmor}) and (\ref{eqn:IKWG}), respectively;
cf. Remark \ref{rem:SCH}.
\begin{proposition}\label{prop:SPRSCH}
Let Assumption \ref{as:semisigmag} hold and set
\begin{equation}\label{def:coeffd}
d(t,x)\triangleq m(t,x)\qquad \text{on} \quad [0,T]\times\mathbb{R}.
\end{equation}
Then the BSPDE of (\ref{eqn:BSPDEY}) admits solutions of the form
 \begin{equation}\label{eqn:SOLY}
 \mathcal{Y}(t,x,y,z)=-\frac{z}{\sigma(t)y-r(t,x)}+\frac{\Psi^{\sigma}(t)}{\sigma^{2}(t)}+\frac{\sigma(t)r_{x}(t,x)-\Psi^{\sigma}(t)y
 +\Psi^{r}(t,x)}{\sigma(t)\big(\sigma(t)y-r(t,x)\big)}
 \end{equation}
 on $[0,T)\times\mathbb{R}^{3}$,
where
\begin{equation}\label{def:ec}
e(t,x)\triangleq\sigma(t)m_{x}(t,x) \  \ \text{and}\  \ c(t,x)\triangleq-\int f(t,x)dx+\bar{c}(t)\quad \text{on}\ [0,T]\times\mathbb{R}
\end{equation}
 for some stochastic process $\bar{c}(\cdot)$, the random field $\,r:[0,T]\times\mathbb{R}\times\Omega\rightarrow \mathbb{R}\,$ satisfies the \emph{linear} BSPDE
\begin{align}\label{eqn:BSPDEr}
dr=\Bigg\{&-\frac{1}{2}\;\sigma^2(t)r_{xx}+\Big(\Psi^{\sigma}(t)+\sigma(t)m(t,x)\Big)r_{x}\\
&+\left[\sigma(t)c(t,x)-\Psi^{\sigma}(t)m(t,x)-\frac{(\Psi^{\sigma}(t))^2}{\sigma(t)}+A^{\sigma}(t)\right]\frac{r}{\sigma(t)}-\sigma(t)\Psi_{x}^{r}\nonumber\\
&+\left(\frac{\Psi^{\sigma}(t)}{\sigma(t)}+m(t,x)\right)\Psi^{r}\Bigg\}dt+\Psi^{r}dW(t)\qquad \text{on}\qquad [0,T)\times\mathbb{R},\nonumber
\end{align}
and the constraint
\begin{align}\label{eqn:bigconstraint}
\qquad\qquad\ &\bigg[-\frac{2(\Psi^{\sigma}(t))^2}{\sigma(t)}+\Psi^{\Psi^{\sigma}}(t)\bigg]r_{x}\nonumber\\
+&\Bigg[5\frac{\big(\Psi^{\sigma}(t)\big)^3}{\sigma^2(t)}
-3\frac{\Psi^{\Psi^{\sigma}}(t)\Psi^{\sigma}(t)}{\sigma(t)}-\Psi^{\Psi^{\sigma}}(t)m(t,x)
\nonumber\\
&\qquad\qquad\  \ \ +\frac{2\big(\Psi^{\sigma}(t)\big)^2m(t,x)}{\sigma(t)}
-\frac{2A^{\sigma}(t)\Psi^{\sigma}(t)}{\sigma(t)}+A^{\Psi^{\sigma}}(t)\Bigg]\frac{r}{\sigma(t)}\nonumber\\
-&\frac{1}{2} \ \sigma^2(t)\Psi_{xx}^{r}
+\Big(2\Psi^{\sigma}(t)+\sigma(t)m(t,x)\Big)\Psi_{x}^{r}\nonumber\\
+&\Bigg[2\frac{A^{\sigma}(t)}{\sigma(t)}-5\frac{(\Psi^{\sigma}(t))^2}{\sigma^2(t)}
+\frac{\Psi^{\Psi^{\sigma}}(t)}{\sigma(t)}+c(t,x)-2\frac{\Psi^{\sigma}(t)m(t,x)}{\sigma(t)}\Bigg]\Psi^{r}\nonumber\\
-&\sigma(t)\Psi_{x}^{\Psi^{r}}
+\bigg(2\frac{\Psi^{\sigma}(t)}{\sigma(t)}+m(t,x)\bigg)\Psi^{\Psi^{r}}
+\frac{1}{2}\;\sigma^3(t)\Psi_{xx}^{V}\\
-&\sigma(t)\Big(\sigma(t)m(t,x)+\Psi^{\sigma}(t)\Big)\Psi_{x}^{V}\nonumber\\
-&\bigg(A^{\sigma}(t)-\Psi^{\sigma}(t)m(t,x)+\sigma(t)c(t,x)
-\frac{(\Psi^{\sigma}(t))^2}{\sigma(t)}\bigg)\Psi^{V}\nonumber\\
+&\sigma(t)A^{\Psi^{V}}+\sigma^2(t)\Psi_{x}^{\Psi^{V}}
-\Big(\sigma(t)m(t,x)+\Psi^{\sigma}(t)\Big)\Psi^{\Psi^{V}}-A^{\Psi^{r}}=0\nonumber
\end{align}
holds on $[0,T)\times\mathbb{R}$.
\end{proposition}

\begin{proof}
Since (\ref{eqn:BSPDEY}) holds for all $h \in
{\mathbb R}$, we see that from the term quadratic in $h$ we get
$\mathcal{Y}_{zz}=0,$ which on the real line shows that
\begin{equation}\label{eqn:LINEARY}
\mathcal{Y}(t,x,y,z)=\mathcal{P}(t,x,y)z+\mathcal{Q}(t,x,y) \qquad \text{on}\qquad [0,T)\times\mathbb{R}^3
\end{equation}
for appropriate random fields $\mathcal{P}, \mathcal{Q}\in C_{\mathbb{F}}\big([0,T];\mathbb{L}^2(\Omega;C^3(\mathbb{R}^2))\big)$, that each obtains a semimartingale decomposition.
Additionally, from the terms linear in $h$ and (\ref{def:coeffd}) we find that
\[ -\sigma(t)\mathcal{Y}_{xz}-z\mathcal{Y}_{yz}+\sigma(t)\mathcal{Y}\mathcal{Y}_{z}
-\frac{\Psi^{\sigma}(t)}{\sigma(t)} \mathcal{Y}_{z}-\Psi^{\mathcal{Y}}_{z}=0;  \]
substituting (\ref{eqn:LINEARY}), the left-hand-side becomes a linear polynomial with respect to $z$
which implies that
\begin{align}\label{eqn:PQ}
\mathcal{P}(t,x,y)&=-\frac{1}{\sigma(t)y-r(t,x)} \qquad \text{and} \\ \nonumber\\
\label{eqn:Qred}\mathcal{Q}(t,x,y)&=\frac{\Psi^{\sigma}(t)}{\sigma^{2}(t)}+\frac{\sigma(t)r_{x}(t,x)-\Psi^{\sigma}(t)y+\Psi^{r}(t,x)}{\sigma(t)\big(\sigma(t)y-r(t,x)\big)}
\end{align}
on $[0,T)\times\mathbb{R}^2$, in terms of a random field $r:[0,T]\times\mathbb{R}\times\Omega\rightarrow\mathbb{R}$
with a semimartingale decomposition to be determined.

As a consequence of this analysis, the BSPDE (\ref{eqn:BSPDEY}) takes the reduced form
\begin{align}
d\mathcal{Y}=\Bigg\{&-\frac{1}{2} \ \sigma^2(t)\mathcal{Y}_{xx}-\frac{1}{2}\ z^2\mathcal{Y}_{yy}-\sigma(t)z\mathcal{Y}_{xy}+\sigma^{2}(t)\mathcal{Y}\mathcal{Y}_{x}
+\sigma(t)z\mathcal{Y}\mathcal{Y}_{y}
\nonumber\\
&-\Bigg[\bigg(A^{\sigma}(t)+\sigma^2(t)m_{x}\big(t,x(t)\big)\\
&\qquad\qquad\quad-\Psi^{\sigma}(t)m\big(t,x(t)\big)
+\sigma(t)c\big(t,x(t)\big)-\frac{(\Psi^{\sigma}(t))^2}{\sigma(t)}\bigg)\frac{z(t)}{\sigma(t)}\nonumber\\
&\qquad+\frac{1}{2}\;\sigma^2(t)\Psi_{xx}^{V}\big(t,x(t)\big)-\bigg(\sigma(t)m\big(t,x(t)\big)
+\Psi^{\sigma}(t)\bigg)\Psi_{x}^{V}\big(t,x(t)\big)\nonumber\\
&\qquad-\bigg(A^{\sigma}(t)-\Psi^{\sigma}(t)m\big(t,x(t)\big)\nonumber\\
&\qquad\qquad\qquad\ +\sigma(t)c\big(t,x(t)\big)-\frac{(\Psi^{\sigma}(t))^2}{\sigma(t)}\bigg)\frac{\Psi^{V}\big(t,x(t)\big)}{\sigma(t)}+A^{\Psi^{{V}}}(t)\nonumber
\end{align}
\begin{align*}
&\qquad +\sigma(t)\Psi_{x}^{\Psi^{V}}\big(t,x(t)\big)-\left(m\big(t,x(t)\big)+\frac{\Psi^{\sigma}(t)}{\sigma(t)}\right)\Psi^{\Psi^{V}}\nonumber\\
&\qquad\qquad\qquad\ \qquad\qquad\qquad\qquad\qquad\ +\sigma(t)c_{x}\big(t,x(t)\big)y(t)\Bigg]\mathcal{Y}_{z}\nonumber\\
&+\sigma(t)m(t,x)\mathcal{Y}_{x}-c(t,x)y\mathcal{Y}_{y}
+e(t,x)\mathcal{Y}+\sigma(t)\mathcal{Y}\Psi^{\mathcal{Y}}
-\sigma(t)\Psi^{\mathcal{Y}}_{x}\nonumber\\
&-z\Psi^{\mathcal{Y}}_{y}+m(t,x)\Psi^{\mathcal{Y}}+f(t,x)\Bigg\}dt+\Psi^{\mathcal{Y}}dW(t)\quad \text{on}\quad [0,T)\times\mathbb{R}^3\;,\nonumber
\end{align*}
and from (\ref{eqn:LINEARY}), (\ref{eqn:PQ}), and (\ref{eqn:Qred}) obtains solutions in the form of (\ref{eqn:SOLY}).
Through substitution of (\ref{eqn:LINEARY})
and comparing the finite-variation terms of the two sides, we arrive again
at an equation between two linear polynomials with respect to $z$ which yields the relationships:
\begin{align}
\quad A^{\mathcal{P}}&=-\frac{1}{2}\ \sigma^2(t)\mathcal{P}_{xx}-\sigma(t)\mathcal{Q}_{xy}+\sigma^{2}(t)\mathcal{Q}\mathcal{P}_{x}+\sigma^{2}(t)\mathcal{P}\mathcal{Q}_{x}
+\sigma(t)\mathcal{Q}\mathcal{Q}_{y}\nonumber\\
&\quad\,+\sigma(t)m(t,x)\mathcal{P}_{x}-c(t,x)y\mathcal{P}_{y}+e(t,x)\mathcal{P}
+\sigma(t)\mathcal{P}\Psi^{\mathcal{Q}}+\sigma(t)\mathcal{Q}\Psi^{\mathcal{P}}
\nonumber\\
&\,\,\,\ \, \,-\bigg[A^{\sigma}(t)+\sigma^2(t)m_{x}(t,x)+\sigma(t)c(t,x)\nonumber\\
&\qquad\qquad\qquad\qquad\qquad\,\,\quad\ \ \, -\Psi^{\sigma}(t)m(t,x)
-\frac{(\Psi^{\sigma}(t))^2}{\sigma(t)}\bigg]\frac{\mathcal{P}}{\sigma(t)}\nonumber\\
&\quad\,-\sigma(t)\Psi^{\mathcal{P}}_{x}
-\Psi^{\mathcal{Q}}_{y}+m(t,x)\Psi^{\mathcal{P}},\nonumber\\
\label{eqn:APAQ} \\
A^{\mathcal{Q}}&=-\frac{1}{2} \ \sigma^2(t)\mathcal{Q}_{xx}+\sigma(t)\bigg(\sigma(t)\mathcal{Q}+m(t,x)\bigg)\mathcal{Q}_{x}-c(t,x)y\mathcal{Q}_{y}+e(t,x)\mathcal{Q}\nonumber\\
&\quad\,-\bigg[\frac{1}{2}\;\sigma^2(t)\Psi_{xx}^{V}\big(t,x(t)\big)-\bigg(\sigma(t)m\big(t,x(t)\big)
+\Psi^{\sigma}(t)\bigg)\Psi_{x}^{V}\big(t,x(t)\big)\nonumber\\
& \ \ \ \qquad-\bigg(A^{\sigma}(t)-\Psi^{\sigma}(t)m\big(t,x(t)\big)\nonumber\\
&\qquad\qquad\qquad\ \ \ \ +\sigma(t)c\big(t,x(t)\big)
-\frac{(\Psi^{\sigma}(t))^2}{\sigma(t)}\bigg)\frac{\Psi^{V}\big(t,x(t)\big)}{\sigma(t)}\nonumber\\
& \ \ \ \qquad +A^{\Psi^{V}}(t)+\sigma(t)\Psi_{x}^{\Psi^{V}}\big(t,x(t)\big)-\left(m\big(t,x(t)\big)
+\frac{\Psi^{\sigma}(t)}{\sigma(t)}\right)\Psi^{\Psi^{V}}\nonumber\\
& \ \ \ \qquad +\sigma(t)c_{x}\big(t,x(t)\big)y\bigg]\mathcal{P}-\sigma(t)\Psi^{\mathcal{Q}}_{x}+\bigg(\sigma(t)\mathcal{Q}+m(t,x)\bigg)\Psi^{\mathcal{Q}}+f(t,x),\nonumber
\end{align}
both on $[0,T)\times\mathbb{R}^2.$
Employ It\^{o} differentials to (\ref{eqn:PQ}) and (\ref{eqn:Qred}) to compute explicitly the semimartingale decompositions of
$\mathcal{P}$ and $\mathcal{Q}$,
and make the substitutions in (\ref{eqn:APAQ}) to end up at two zero-polynomials in $y$, of first and second order respectively,
which imply eventually the definitions of (\ref{def:ec}), the BSPDE (\ref{eqn:BSPDEr}), and the constraint (\ref{eqn:bigconstraint}).
\end{proof}

BSPDE (\ref{eqn:BSPDEr}) admits the apparent solution of $r(\cdot\, ,\cdot)=0$; then, (\ref{eqn:SOLY}) simplifies to
\begin{equation}\label{eqn:DEFcalY}
\mathcal{Y}(t,x,y,z)=-\frac{z}{\sigma(t)y} \qquad \text{on}\qquad [0,T)\times\mathbb{R}^3,
\end{equation}
as well as, solving (\ref{eqn:bigconstraint}) for $A^{\Psi^{V}}$ and substituting back to the semimartingale decomposition of $\Psi^{V}$,  we get the BSPDE constraint
\begin{align}\label{eqn:midconstraint}
d\Psi^{V}=\bigg\{&-\frac{1}{2}\;\sigma^2(t)\Psi_{xx}^{V}
+\Big[\sigma(t)m(t,x)+\Psi^{\sigma}(t)\Big]\Psi_{x}^{V}\nonumber\\
&+\left(\frac{A^{\sigma}(t)}{\sigma(t)}-\frac{\Psi^{\sigma}(t)}{\sigma(t)}m(t,x)+c(t,x)-\frac{(\Psi^{\sigma}(t))^2}{\sigma^2(t)}\right)\Psi^{V}\\
&-\sigma(t)\Psi_{x}^{\Psi^{V}}+\left(\frac{\Psi^{\sigma}(t)}{\sigma(t)}+m(t,x)\right)\Psi^{\Psi^{V}}\bigg\}dt+\Psi^{\Psi^{V}}dW(t).\nonumber
\end{align}
on $[0,T)\times\mathbb{R}.$
Moreover, combining (\ref{def:triplBSB}), (\ref{def:triplBSH}), (\ref{eqn:STRANSF}), and (\ref{eqn:DEFcalY}) we have that
\begin{align}\label{eqn:relatUandV}
\quad U\big(t,x(t)\big)=Y(t)=\mathcal{Y}\big(t,x(t),y(t),z(t)\big)&=-\frac{z(t)}{\sigma(t)y(t)}\\ \nonumber\\
&=-\frac{V_{x}\big(t,x(t)\big)}{V\big(t,x(t)\big)}
-\frac{\Psi^{V}\big(t,x(t)\big)}{\sigma(t)V\big(t,x(t)\big)}\nonumber,
\end{align}
 which is always the case if
 \begin{equation}\label{eqn:SCH}
 U=-\frac{\partial}{\partial x}\ln V-\frac{\Psi^{V}}{\sigma(t)V}
 \end{equation}
holds for every $(t,x)\in[0,T)\times\mathbb{R}$. In fact for $t=T$, putting together the terminal conditions of (\ref{eqn:SBURGERSredmor}) and (\ref{eqn:IKWG}), we get that the random field $q$ must satisfy the condition
 \begin{equation}\label{eqn:relatpandq}
 p(x)=-\frac{\partial}{\partial x}\ln q(x)-\frac{\Psi^{V}(T,x)}{\sigma(T)q(x)}, \qquad x\in\mathbb{R}.
 \end{equation}

 \begin{remark}\label{rem:SCH}
 The relationship of (\ref{eqn:SCH}) is
 the \emph{generalized} Cole-Hopf transformation that provides the pattern to construct solutions $(U,\Psi^{U})$ of the
stochastic Burgers BSPDE (\ref{eqn:SBURGERSredmor}) from solutions $(V,\Psi^{V})$ of the stochastic linear heat equation (\ref{eqn:IKWG}), subject to the additional
\emph{linear} constraint of (\ref{eqn:midconstraint}) and the terminal condition of (\ref{eqn:relatpandq}); a \emph{verification} result is provided below (cf. Theorem \ref{thm:SCH}). Apparently in a deterministic setting (\ref{eqn:SCH}) reduces to the standard Cole-Hopf transformation of Section \ref{sec:PACH} and (\ref{eqn:midconstraint}) becomes a tautology.

One can immediately see that considering either deterministic or random coefficients the point transformations $\mathcal{Y}$
of (\ref{eqn:TRANSF}) and (\ref{eqn:STRANSF}) are both of the same form; cf. Remark \ref{rem:DRCH} and (\ref{eqn:DEFcalY}). However, the application of the
probabilistic characterization of the Cole-Hopf transformation, deployed in the previous section, led to Proposition \ref{prop:PRSCH} and eventually concluded to
the additional constraint for the random field $\Psi^{V}$ of (\ref{eqn:midconstraint}), which vanishes in the deterministic case.
\end{remark}

\begin{theorem}\label{thm:SCH}
Consider Assumption \ref{as:semisigmag}, definitions (\ref{def:coeffd}) and (\ref{def:ec}), and a random field pair $(V,\Psi^{V})\in C_{\mathbb{F}}\big([0,T];\mathbb{L}^2(\Omega;C^4(\mathbb{R}))\big)\times \mathbb{L}_{\mathbb{F}}^2\big(0,T;C^3(\mathbb{R})\big)$ that satisfies the heat-type BSPDE (\ref{eqn:IKWG}) with $V(\cdot\, ,\cdot)>0$, subject to the constraints (\ref{eqn:midconstraint}) for the pair $(\Psi^{V},\Psi^{\Psi^{V}})$ of class $C_{\mathbb{F}}\big([0,T];\mathbb{L}^2(\Omega;C^3(\mathbb{R}))\big)\times \mathbb{L}_{\mathbb{F}}^2\big(0,T;C^2(\mathbb{R})\big)$ and (\ref{eqn:relatpandq}) for the random field $q$. Then the random field pair $(U,\Psi^{U})\in C_{\mathbb{F}}\big([0,T];\mathbb{L}^2(\Omega;C^3(\mathbb{R}))\big)\times \mathbb{L}_{\mathbb{F}}^2\big(0,T;C^2(\mathbb{R})\big)$ defined through (\ref{eqn:SCH}) is a solution of the Burgers-type BSPDE  (\ref{eqn:SBURGERSredmor}).
\end{theorem}

\begin{remark}
It is important at this point to comment on the constraints as expressed in equations (\ref{eqn:midconstraint}) and (\ref{eqn:relatpandq}).
The backward heat equation (\ref{eqn:IKWG}) admits a unique solution pair $(V,\Psi^{V})$. Then, the uniqueness of the semimartingale decomposition for the random field $\Psi^{V}$ provides us with a uniquely determined $\Psi^{\Psi^{V}}$.
In Proposition \ref{cor:solvableburg} below we present two special cases of heat-type BSPDE (\ref{eqn:IKWG}), whose solution satisfies these constraints for given coefficients of the original Burgers-type BSPDE (\ref{eqn:SBURGERSredmor}).

In general, for a given final condition $q$ a fruitful way of handling these constraints is to substitute in them the pair $(\Psi^{V},\Psi^{\Psi^{V}}),$  as defined by the solution of (\ref{eqn:IKWG}),
 and then find the general class of coefficients of the original problem so that the constraints are satisfied.
 This approach provides us with large families of exactly solvable backward Burgers equations with random coefficients; see Subsection \ref{subsec:EXAMPLESCONSTRAINTS}.
\end{remark}

\begin{remark}
It is worth mentioning that, regarding (\ref{def:coeffd}), (\ref{def:ec}), and solution pairs   $(V,\Psi^{V})$ and  $(r,\Psi^{r})$ of the BSPDEs (\ref{eqn:IKWG}) and (\ref{eqn:BSPDEr}), respectively, the random field $\mathcal{Y}$ of (\ref{eqn:STRANSF}) defined through (\ref{eqn:SOLY})
provides the most general form of solution to  (\ref{eqn:BSPDEY}), subject to the  constraint (\ref{eqn:bigconstraint}). In view of
(\ref{eqn:relatUandV}) and (\ref{eqn:SCH}) we obtain here the relationship
\begin{equation}
U=-\frac{V_{x}}{V-r}
-\frac{\Psi^{V}}{\sigma(t)V-r}
+\frac{\Psi^{\sigma}(t)}{\sigma^{2}(t)}+\frac{\sigma(t)r_{x}-\Psi^{\sigma}(t)V+\Psi^{r}}{\sigma(t)\big(\sigma(t)V-r\big)}\nonumber
\end{equation}
for every $(t,x)\in[0,T)\times\mathbb{R},$ which establishes the \emph{most general formulation} of the Cole-Hopf transformation
between the BSPDEs of (\ref{eqn:SBURGERSredmor}) and  (\ref{eqn:IKWG}).
\end{remark}

\begin{proof} \emph{\textsc{of theorem \ref{thm:SCH}}}:
For $0\leq t<T$ and $x\in\mathbb{R}$, It\^{o}'s rule, (\ref{eqn:SCH}), (\ref{eqn:IKWG}), and  (\ref{eqn:midconstraint}) yield
\begin{align}\label{eqn:VerU}
dU=\Bigg\{&\sigma(t)\frac{\Psi_{xx}^{V}}{2V}-\Big(\sigma(t)m(t,x)+\Psi^{\sigma}(t)\Big)\frac{\Psi_{x}^{V}}{\sigma(t)V}
+\Psi^{\sigma}(t)m(t,x)\frac{\Psi^{V}}{\sigma^2(t)V}\nonumber\\
&+\frac{\Psi_{x}^{\Psi^{V}}}{V}-m(t,x)\frac{\Psi^{\Psi^{V}}}{\sigma(t)V}-\sigma(t)\frac{\Psi^{V}V_{xx}}{2V^2}+m(t,x)\frac{\Psi^{V}V_{x}}{V^2}\\
&+m(t,x)\frac{(\Psi^{V})^2}{\sigma(t)V^2}-\frac{\Psi^{V}\Psi_{x}^{V}}{V^2}
-\frac{(\Psi^{V})^3}{\sigma(t)V^3}-\Psi^{\sigma}(t)\frac{(\Psi^{V})^2}{\sigma^2(t)V^2}
+\frac{\Psi^{V}\Psi^{\Psi^{V}}}{\sigma(t)V^2}\nonumber\\
&+\sigma^2(t)\frac{V_{xxx}}{2V}-\sigma^2(t)\frac{V_{xx}V_{x}}{2V^2}
-\sigma(t)m_{x}(t,x)\frac{V_{x}}{V}-\sigma(t)m(t,x)\frac{V_{xx}}{V}\nonumber\\
&+\sigma(t)m(t,x)\frac{(V_{x})^2}{V^2}-m_{x}(t,x)\frac{\Psi^{V}}{V}-m(t,x)\frac{\Psi_{x}^{V}}{V}
+m(t,x)\frac{\Psi^{V}V_{x}}{V^2}\nonumber\\
&+\sigma(t)\frac{\Psi_{xx}^{V}}{V}-\sigma(t)\frac{\Psi_{x}^{V}V_{x}}{V^2}-c_{x}(t,x)
+\frac{\Psi_{x}^{V}\Psi^{V}}{V^2}-\frac{(\Psi^{V})^2V_{x}}{V^3}\Bigg\}dt\nonumber\\
+&\Bigg\{-\frac{\Psi^{\Psi^{V}}}{\sigma(t)V}+\frac{(\Psi^{V})^2}{\sigma(t)V^2}
+\Psi^{\sigma}(t)\frac{\Psi^{V}}{\sigma^2(t)V}-\frac{\Psi_{x}^{V}}{V}+\frac{\Psi^{V}V_{x}}{V^2}\Bigg\}dW(t).\nonumber
\end{align}
Setting
\begin{align*}
\Psi^{U}&=-\frac{\Psi^{\Psi^{V}}}{\sigma(t)V}+\frac{(\Psi^{V})^2}{\sigma(t)V^2}
+\Psi^{\sigma}(t)\frac{\Psi^{V}}{\sigma^2(t)V}-\frac{\Psi_{x}^{V}}{V}+\frac{\Psi^{V}V_{x}}{V^2},\\
\end{align*}
in conjunction with (\ref{eqn:SCH}), it is straightforward to obtain that

\begin{align*}
\Psi_{x}^{U}&=-\frac{\Psi_{x}^{\Psi^{V}}}{\sigma(t)V}+\frac{\Psi^{\Psi^{V}}V_{x}}{\sigma(t)V^2}+2\frac{\Psi^{V}\Psi_{x}^{V}}{\sigma(t)V^2}
-2\frac{(\Psi^{V})^2V_{x}}{\sigma(t)V^3}+\Psi^{\sigma}(t)\frac{\Psi_{x}^{V}}{\sigma^2(t)V}\\
\\
&\quad\,-\Psi^{\sigma}(t)\frac{\Psi^{V}V_{x}}{\sigma^2(t)V^2}
-\frac{\Psi_{xx}^{V}}{V} +2\frac{\Psi_{x}^{V}V_{x}}{V^2}+\frac{\Psi^{V}V_{xx}}{V^2}-2\frac{\Psi^{V}(V_{x})^2}{V^3},\\
\\
U_{x}&=-\frac{\Psi_{x}^{V}}{\sigma(t)V}+\frac{\Psi^{V}V_{x}}{\sigma(t)V^2}-\frac{V_{xx}}{V}+\frac{(V_{x})^2}{V^2},
\end{align*}
\begin{align*}
UU_{x}&=\frac{\Psi^{V}\Psi_{x}^{V}}{\sigma^2(t)V^2}-\frac{(\Psi^{V})^2V_{x}}{\sigma^2(t)V^3}+\frac{\Psi^{V}V_{xx}}{\sigma(t)V^2}
+\frac{\Psi_{x}^{V}V_{x}}{\sigma(t)V^2}-2\frac{\Psi^{V}(V_{x})^2}{\sigma(t)V^3}\\ \\
&\quad\,+\frac{V_{xx}V_{x}}{V^2}-\frac{(V_{x})^3}{V^3},\\
\\
U\Psi^{U}&=\frac{\Psi^{V}\Psi^{\Psi^{V}}}{\sigma^2(t)V^2}-\frac{(\Psi^{V})^3}{\sigma^2(t)V^3}-\Psi^{\sigma}(t)\frac{(\Psi^{V})^2}{\sigma^3(t)V^2}
+\frac{\Psi^{V}\Psi_{x}^{V}}{\sigma(t)V^2}+\frac{\Psi^{\Psi^{V}}V_{x}}{\sigma(t)V^2}\\ \\
&\quad\,-2\frac{(\Psi^{V})^2V_{x}}{\sigma(t)V^3}+\Psi^{\sigma}(t)\frac{\Psi^{V}V_{x}}{\sigma^2(t)V^2}
+\frac{\Psi_{x}^{V}V_{x}}{V^2}-\frac{(V_{x})^2\Psi^{V}}{V^3},\\
\\
U_{xx}&=-\frac{\Psi_{xx}^{V}}{\sigma(t)V}
+\frac{\Psi^{V}V_{xx}}{\sigma(t)V^2}-\frac{V_{xxx}}{V}+3\frac{V_{xx}V_{x}}{V^2}-2\frac{(V_{x})^3}{V^3}\nonumber\\
\\
&\quad\,+2\frac{\Psi_{x}^{V}V_{x}}{\sigma(t)V^2}
-2\frac{(V_{x})^2\Psi^{V}}{\sigma(t)V^3}.
\end{align*}
Therefore, recalling (\ref{def:coeffd}) and (\ref{def:ec}), direct substitution of the above into (\ref{eqn:SBURGERSredmor}) leads to (\ref{eqn:VerU}) and completes the proof.
\end{proof}
\smallskip

Theorem \ref{thm:SCH} necessitates the investigation of existence, uniqueness, and positivity of solutions for the linear BSPDE (\ref{eqn:IKWG}) subject to the constraints (\ref{eqn:midconstraint}), which is another linear BSPDE for their martingale part,  and (\ref{eqn:relatpandq}). Ma \& Yong \big(\cite{Ma-Yong97}, \cite{Ma-Yong99}\big) studied linear BSPDEs of parabolic type
and provided regularity conditions on their coefficients that ensured existence, uniqueness, and comparison results between solutions. However, their results are not applicable in the case considered here on account of the  presence of the BSPDE constraint.

\begin{proposition}\label{cor:solvableburg}
Let Assumption \ref{as:semisigmag} hold.

 (i) If $V \in C_{\mathbb{F}}\big([0,T];\mathbb{L}^2(\Omega;C^4(\mathbb{R}))\big)$ is a positive random field that satisfies the BSPDE
\begin{align}\label{eqn:stocheatcase1}
dV&=\left[-\frac{1}{2}\,\sigma^2(t) V_{xx}+\sigma(t)m(t,x)V_{x}+c(t,x)V\right]dt,\quad 0\leq t<T,\ x\in \mathbb{R},\\
V(T,x)&=q(x),\qquad x\in \mathbb{R}\nonumber,
\end{align}
where the random field $c$ is given by (\ref{def:ec}) and $q$ is defined as in (\ref{eqn:BSPDESH}),
then the random field $$U=-\frac{\partial}{\partial x} \ln V \in C_{\mathbb{F}}\big([0,T];\mathbb{L}^2(\Omega;C^3(\mathbb{R}))\big),$$
defined through the Cole-Hopf transformation, solves the Burgers-type BSPDE
\begin{align*}
 dU&=\bigg[-\frac{1}{2}\;    \sigma^2(t)U_{xx}+\sigma^{2}(t)UU_{x}+\sigma(t)m(t,x)U_{x}+\sigma(t)m_{x}(t,x)U\nonumber\\
 &\quad\,\,\,\,\,+f(t,x)\bigg]dt,\ 0\leq t<T,\ x\in \mathbb{R},\\
U(T,x)&=p(x),\qquad x\in \mathbb{R}.
 \end{align*}
  (ii) For every $\lambda>-\frac{1}{2}$ with $\lambda\neq 0,$  if $V \in C_{\mathbb{F}}\big([0,T];\mathbb{L}^2(\Omega;C^4(\mathbb{R}))\big)$ is a positive random field that satisfies the BSPDE
\begin{align}\label{eqn:stocheatcase2}
dV&=\Bigg[-\left(\frac{ 1}{2}+\lambda\right)\,\sigma^2(t) V_{xx}+\big(1+\lambda\big)\sigma(t)m(t)V_{x}+c(t)V\Bigg]dt\nonumber\\ 
&\quad+\lambda \sigma(t)V_{x}dW(t),\quad 0\leq t<T,\ x\in \mathbb{R}, \nonumber\\\\
V(T,x)&=q(x),\qquad x\in \mathbb{R},\nonumber\\ \nonumber
\end{align}
where

 $$q(x)\triangleq e^{-\frac{1}{1+\lambda}\int p(x)dx},$$
then the random field $$U=-(1+\lambda)\frac{\partial}{\partial x} \ln V \in C_{\mathbb{F}}\big([0,T];\mathbb{L}^2(\Omega;C^3(\mathbb{R}))\big),$$
defined through the generalized Cole-Hopf transformation of (\ref{eqn:SCH}), solves the Burgers-type BSPDE
\begin{align*}
 dU&=\Bigg[-\left(\frac{1}{2}+\lambda\right)\,\sigma^2(t)U_{xx}+(1+\lambda)\sigma^{2}(t)UU_{x}+\big(1+\lambda\big)\sigma(t)m(t)U_{x}\Bigg]dt\nonumber\\
 &\quad\,+\lambda \sigma(t) U_{x}dW(t),\ 0\leq t<T,\ x\in \mathbb{R},\\\\
U(T,x)&=p(x),\qquad x\in \mathbb{R}.
 \end{align*}
\end{proposition}

\begin{proof}
Both statements of the proposition follow as a direct application of Theorem \ref{thm:SCH} to the random field pairs
$(V,0)$ and $(V,\lambda\sigma(\cdot)V_{x})$, respectively. Specifically, in statement (i)
the constraint (\ref{eqn:midconstraint}) holds trivially. On the other hand, in statement (ii) the random fields $m$ and $c$
are now stochastic processes, i.e. $f=0$ in (\ref{def:ec}), and since $V$ is a solution of
BSPDE (\ref{eqn:stocheatcase2}) then it is straightforward to see by employing the product rule that the random field $\Psi^{V}=\lambda\sigma(\cdot)V_{x}$ satisfies
the constraint (\ref{eqn:midconstraint}).
\end{proof}

\begin{remark}
Regarding the solvability of the BSPDEs (\ref{eqn:stocheatcase1}) and (\ref{eqn:stocheatcase2}), we may assume that the coefficients
 $\sigma,\ m,\ c$ are measurable with respect to the two-sided Wiener process (cf. \cite{Pardoux-Peng}). Then one can invert time
 and arrive in each case at a linear parabolic FSBDE of the form of (\ref{eqn:BSPDESH}), which is solvable according to Theorem
 \ref{thm:shdhsc}.
\end{remark}

\subsection{Examples of completely solvable backward random Burgers equations}\label{subsec:EXAMPLESCONSTRAINTS}

 In what follows, we present examples of random field pairs $(V,\Psi^{V})$ that satisfy the hypotheses of Theorem \ref{thm:SCH}
for \emph{appropriate} coefficients.

\begin{example}\label{ex:SolBSPDEwithC}{\rm
 We postulate positive solutions for (\ref{eqn:IKWG}) of the form

\begin{equation}\label{def:solV1}
V(t,x)\triangleq e^{f_{1}(x)W(t)}
\end{equation}
on $[0,T]\times\mathbb{R}$, in terms of a given function $f_{1}:\mathbb{R}\rightarrow\mathbb{R}$ of class $C(\mathbb{R}^4)$.
From It\^{o}'s formula we have
\begin{equation}\label{eqn:ItoV1}
dV=f_{1}(x)VdW(t)+\frac{1}{2}f_{1}^2(x)Vdt,
\end{equation}
which shows that

\begin{equation}\label{eqn:solMPV1}
\Psi^{V}=f_{1}(x)V.
\end{equation}
Then, the terminal condition of (\ref{eqn:IKWG}) indicates that
the random field $q$ is determined by

\begin{equation}\label{def:termqV1}
q(x)= e^{f_{1}(x)W(T)}, \qquad x\in\mathbb{R},
\end{equation}
thus the random field $p$ of (\ref{eqn:SBURGERSredmor}) must be chosen according to (\ref{eqn:relatpandq}).
Using (\ref{def:solV1}) and (\ref{eqn:solMPV1}), we derive that
\begin{align*}
V_{x}&=f_{1}'(x)VW(t),\\
V_{xx}&=f_{1}''(x)VW(t)+\big(f_{1}'(x)\big)^2VW^2(t),\\
\Psi_{x}^{V}&=f_{1}'(x)V+f_{1}(x)f_{1}'(x)VW(t),\\
\Psi_{xx}^{V}&=f_{1}''(x)V+2\big(f_{1}'(x)\big)^2VW(t)+f_{1}(x)f_{1}''(x)VW(t)\nonumber\\
&\,\,\,\,\,+f_{1}(x)\big(f_{1}'(x)\big)^2VW^2(t),
\end{align*}
and substituting them into (\ref{eqn:IKWG}), in comparison with (\ref{eqn:ItoV1}), we get
\begin{align}\label{eqn:Coeff1V1}
\frac{1}{2}f_{1}^2(x)=&-\frac{1}{2}\,\sigma^2(t)\Big[f_{1}''(x)+(f_{1}'(x))^2W(t)\Big]W(t)+\sigma(t)m(t,x)f_{1}'(x)W(t)\\
&+m(t,x)f_{1}(x)-\sigma(t)f_{1}'(x)[1+f_{1}(x)W(t)]+c(t,x).\nonumber
\end{align}
Furthermore, It\^{o}'s formula and (\ref{eqn:solMPV1}) give

\begin{equation}\label{eqn:ItoMPV1}
d\Psi^{V}=f_{1}^2(x)VdW(t)+\frac{1}{2}f_{1}^3(x)Vdt,
\end{equation}
which implies that

\begin{align*}
\Psi^{\Psi^{V}}&=f_{1}^2(x)V,\\
\Psi_{x}^{\Psi^{V}}&=2f_{1}(x)f_{1}'(x)V+f_{1}^2(x)f_{1}'(x)VW(t);
\end{align*}
then, substitute these equations into (\ref{eqn:midconstraint}) and compare it with (\ref{eqn:ItoMPV1}) to get

\begin{align}\label{eqn:Coeff2V1}
\quad\quad\frac{1}{2}f_{1}^3(x)=&-\frac{1}{2}\,\sigma^2(t)\Big[f_{1}''(x)+2(f_{1}'(x))^2W(t)\nonumber\\
&\qquad\qquad\qquad\quad\, +f_{1}(x)f_{1}''(x)W(t)+f_{1}(x)(f_{1}'(x))^2W^2(t)\Big]\nonumber\\
&+\big(\sigma(t)m(t,x)+\Psi^{\sigma}(t)\big)f_{1}'(x)\big[1+f_{1}(x)W(t)\big]\\
&+\left(\frac{A^{\sigma}(t)}{\sigma(t)}-\frac{\Psi^{\sigma}(t)}{\sigma(t)}m(t,x)+c(t,x)-\frac{(\Psi^{\sigma}(t))^2}{\sigma^2(t)}\right)f_{1}(x)\nonumber\\
&-\sigma(t)f_{1}(x)f_{1}'(x)\big[2+f_{1}(x)W(t)\big]+\bigg(\frac{\Psi^{\sigma}(t)}{\sigma(t)}+m(t,x)\bigg)f_{1}^2(x).\nonumber
\end{align}

Obviously the random field triplet  $(V,\Psi^{V},\Psi^{\Psi^{V}})$ belongs to the class $ C_{\mathbb{F}}\big([0,T];\mathbb{L}^2(\Omega;C^4(\mathbb{R}))\big)\times C_{\mathbb{F}}\big([0,T];\mathbb{L}^2(\Omega;C^3(\mathbb{R}))\big) \times \mathbb{L}_{\mathbb{F}}^2\big(0,T;C^2(\mathbb{R})\big),$ according to Theorem  \ref{thm:SCH}. Furthermore, equations (\ref{eqn:Coeff1V1}) and (\ref{eqn:Coeff2V1}) form a $2\times2$ linear system which can be easily solved for the random fields $m$ and $c$. A selection of these coefficients together with (\ref{def:termqV1}) constitute the pair $(V,\Psi^{V})$ of (\ref{def:solV1}) and (\ref{eqn:solMPV1}) as a positive solution of (\ref{eqn:IKWG}) subject to the constraint (\ref{eqn:midconstraint}).
}
\end{example}

\begin{example}\label{ex:SolBSPDEwithC2}{\rm
 Alternatively, for a given positive function $f_{2}:\mathbb{R}\rightarrow\mathbb{R}^{+}$ of class $C(\mathbb{R}^4)$, we consider that (\ref{eqn:IKWG}) admits a solution of the form
\begin{equation}\label{def:solV2}
V(t,x)\triangleq f_{2}(x)e^{W(t)} \in  C_{\mathbb{F}}\big([0,T];\mathbb{L}^2(\Omega;C^4(\mathbb{R}))\big).
\end{equation}
 Now, It\^{o}'s rule yields
\begin{equation}\label{eqn:ItoV2}
dV=VdW(t)+\frac{1}{2}\,Vdt\qquad \text{and} \qquad \Psi^{V}=V,
\end{equation}
the terminal condition of (\ref{eqn:IKWG}) imposes that
the random field $q$ must be given by
\begin{equation}\label{def:termqV2}
q(x)= f_{2}(x)e^{W(T)}, \qquad x\in\mathbb{R},
\end{equation}
and the constraint (\ref{eqn:relatpandq}) designates the random field $p$ of (\ref{eqn:SBURGERSredmor}) that should be considered.
Therefore, differentiating (\ref{def:solV2}) we obtain
\begin{align*}
V_{x}=f_{2}'(x)e^{W(t)}\qquad \text{and}\qquad
V_{xx}=f_{2}''(x)e^{W(t)},
\end{align*}
and through substitution back to (\ref{eqn:IKWG}) and (\ref{eqn:midconstraint}), a comparison with (\ref{eqn:ItoV2}) reveals that

\begin{align}
\frac{1}{2}f_{2}(x)=&-\frac{1}{2}\,\sigma^2(t)f_{2}''(x)+\sigma(t)m(t,x)f_{2}'(x)+m(t,x)f_{2}(x)\nonumber\\
&-\sigma(t)f_{2}'(x)+c(t,x)f_{2}(x),\nonumber\\
\frac{1}{2}f_{2}(x)=&-\frac{1}{2}\sigma^2(t)f_{2}''(x)
+\Big(\sigma(t)m(t,x)+\Psi^{\sigma}(t)\Big)f_{2}'(x)\nonumber\\
&+\left(\frac{A^{\sigma}(t)}{\sigma(t)}-\frac{\Psi^{\sigma}(t)}{\sigma(t)}m(t,x)
+c(t,x)-\frac{(\Psi^{\sigma}(t))^2}{\sigma^2(t)}\right)f_{2}(x)
\nonumber\\
&-\sigma(t)f_{2}'(x)+\left(\frac{\Psi^{\sigma}(t)}{\sigma(t)}+m(t,x)\right)f_{2}(x)\nonumber
\end{align}
should hold respectively. Similarly to (i), the solution of this $2\times2$ linear system determines the proper random fields $m$ and $c$ for which
the pair $(V,\Psi^{V})$ of (\ref{def:solV2}) and (\ref{eqn:ItoV2}) is a positive solution of (\ref{eqn:IKWG}) subject to the constraint (\ref{eqn:midconstraint}).}
\end{example}

\section{Stochastic Feynman-Kac Formula}\label{sec:SFKF}
In this section we establish \emph{stochastic} Feynman-Kac results for the solutions of the Burgers FSPDE (\ref{eqn:forwardSBURGERSredmore}) and BSPDE (\ref{eqn:SBURGERSredmor});
 cf. Propositions  \ref{prop:FFKU} and \ref{prop:BFKU} below, respectively. Both results rely on the Cole-Hopf transformation, in its simple or generalized version (cf. (\ref{eqn:SCH}))
 respectively, which reduces the initial SPDEs to linear ones.
\begin{proposition}\label{prop:FFKU}
 Let the assumptions of Theorem \ref{thm:shdhsc} hold, let the processes $\sigma$, $\bar{c}$ and the random fields $b, \, m$ be constants, let
 $f(\,\cdot\,,\cdot)=0$, as well as $p\in C^{2}(\mathbb{R})$. If, in addition, there exist constants $K>0$ and $0<\alpha<1/2T$ such that
 $$\max_{0\leq t \leq T}|G(t,x)|+\max_{0\leq t \leq T}|G_{x}(t,x)|\leq K e^{\alpha x^{2}} \qquad \forall\ x\in \mathbb{R}, $$
 then the solution  random field $U$ of the corresponding Burgers FSPDE (\ref{eqn:forwardSBURGERSredmore})
 admits for every $(t,x)\in[0,T]\times\mathbb{R}$ and a.s. the stochastic representation
 \begin{equation*}
 U(t,x)=-\frac{G_{x}\big(t,x+H(t)\big)}{G\big(t,x+H(t)\big)}\, ,
 \end{equation*}
 where
 \begin{equation*}
 G(t,x)=e^{-\bar{c}\,t}E^{x}\Big[q\big(x(t)\big)\Big] \qquad \text{and}\qquad
 G_{x}(t,x)=e^{-\bar{c}\,t}E^{x}\Big[q_{x}\big(x(t)\big)\Big];
 \end{equation*}
 here, the notation $E^{x}$ stands for the expectation that corresponds to the process
\begin{align*}
dx(t)&=k\,dt+\sigma \, dW(t), \qquad 0<t\leq T, \\
x(0)&=x.
\end{align*}
\end{proposition}
\begin{proof}
According to hypothesis, differentiation of (\ref{eqn:BSPDEG}) implies that the random field $G_{x}$ solves the same FSPDE with $G$ but with
initial condition given by $G_{x}(0,x)=q_{x}(x)$ for $x\in \mathbb{R}.$ Then the conclusion of the proposition is obtained by the Cole-Hopf transformation, (\ref{def:SOLV}), and standard
Feynman-Kac results (e.g. \cite{Karatzas-Shreve} Corollary 4.4.5) applied to the FSPDEs whose solutions are the random fields $G$ and $G_{x}$.
\end{proof}
\begin{proposition}\label{prop:BFKU}
Consider the FBSDEs of (\ref{eqn:FBSDESH}) on the interval $[t,T]$ with $x(t)=x$ for some $0\leq t \leq T,$ and assume that the hypotheses of Theorem \ref{thm:SCH} hold
 with $m(\,\cdot\, ,\cdot)=d(\,\cdot\,,\cdot)=\Psi^{\sigma}(\cdot)=0$.
 Then the solution  $(U,\Psi^{U})$ of the Burgers BSPDE (\ref{eqn:SBURGERSredmor}), corresponding to (\ref{eqn:SCH}), obtains for all $(t,x)\in[0,T]\times\mathbb{R}$ and a.s. the representation
 \begin{equation*}
 U(t,x)=-\frac{z^{x}(t)}{\sigma(t)y^{x}(t)}\, ,
 \end{equation*}
 where
 \begin{align*}
 z^{x}(t)&=E^{x}\Bigg[-e^{-\int_{t}^{T}\left(\frac{A^{\sigma}(s)}{\sigma(s)}+c\big(s,x(s)\big)\right)ds}
 \sigma(T)p(x\big(T)\big)q\big(x(T)\big)\\
 & \qquad \quad\,  +\int_{t}^{T}e^{-\int_{t}^{s}\left(\frac{A^{\sigma}(\theta)}{\sigma(\theta)}+c\big(\theta,x(\theta)\big)\right)d\theta}
 \sigma(s)f\big(s,x(s)\big)y^{x}(s)ds\Bigg{|}\mathcal{F}(t)\Bigg],\\
 y^{x}(t)&=E^{x}\left[q\big(x(T)\big)e^{-\int_{t}^{T}c(s,x(s))ds}\Big{|}\mathcal{F}(t)\right];
 \end{align*}
 here, the process $\sigma(\cdot)$ satisfies Assumption \ref{as:semisigmag}, the random field $c$ is given by (\ref{def:ec}), and we denote by $E^{x}$ the expectation that corresponds to the forward process of the aforementioned FBSDE system.
\end{proposition}
\begin{proof}
 Thanks to hypothesis and the constraint (\ref{eqn:midconstraint}), the BSDEs of (\ref{eqn:FBSDESH}) and  (\ref{eqn:SMDz}) simplify to
\begin{align*}
dy(s)&=c\big(s,x(s)\big)y(s)ds+z(s)dW(s),\quad t\leq s<T, \qquad \text{and}\\
dz(s)&=\bigg[\left(\frac{A^{\sigma}(s)}{\sigma(s)}+c\big(s,x(s)\big) \right)z(s)-\sigma(s)f\big(s,x(s)\big)y(s)\bigg]ds+h(s)dW(s),
\end{align*}
for $t\leq s<T,$ respectively. Denoting their solutions by
$y^{x}(\cdot)$ and $z^{x}(\cdot)$, observe that these BSDEs are linear with drift terms independent
 of their auxiliary processes $z(\cdot)$ and $h(\cdot)$, respectively.
Additionally, set $t=T$ in (\ref{eqn:relatUandV}) and invoke the terminal conditions of (\ref{eqn:SBURGERSredmor}) and (\ref{eqn:FBSDESH}) to verify that the terminal condition of the last BSDE above is given by
\begin{equation}\label{eqn:terminalz}
z(T)=-\sigma(T)p(x\big(T)\big)q\big(x(T)\big).
 \end{equation}
Therefore, the assertion of the proposition follows as an application of Corollary 6.2 by Ma \& Yong (1997), in combination with (\ref{eqn:relatUandV}), the above BSDEs, and their terminal conditions in (\ref{eqn:FBSDESH}) and (\ref{eqn:terminalz}).
\end{proof}

\section{Applications} \label{sec:AFB}

The following applications illustrate the use of Theorem  \ref{thm:SCH} to construct solutions of  backward stochastic Burgers equations with random coefficients, which are less commonly found in the literature. These equations turn out to be associated with problems on controllability or mathematical finance.

\subsection{Controllability of a backward random Burgers equation}
Consider the following control system
\begin{align}
\label{eqn:SBURGERSCONTR}
dU&=\Big[-\frac{1}{2} \ \sigma^2(t,x)U_{xx}+a(t,x)UU_{x}+g(t,x)U
\, U_{1} +b(t,x)U_{x}+e(t,x)U\\
& \qquad +s(t,x) U_{\,2} +m(t,x)U_{1} +f(t,x)\Big]dt+U_{1} \, dW(t), \  0\leq t < T,\ x\in \mathbb{R},\nonumber
\end{align}
where $(U_1,U_{\,2})$ is a pair of control processes  to be determined and $p_{0}(x)$ is the initial state of the system. The question we wish to answer is the following. Setting a desired final condition $p(x)$  for the control system (\ref{eqn:SBURGERSCONTR}), for which initial state $p_{0}(x)$ does there exist a pair of control processes $(U_{1},U_{\,2})$ such that this system can be driven in time $T$ to the desired state $p(x)$?

The linearization results for the backward Burgers equation may prove useful to answering this question. Assume that the two control procedures are connected by $U_{\,2}=\frac{\partial}{\partial x} U_{1}$. Then the control system is equivalent to the BSPDE (\ref{eqn:SBURGERS0}) with $U_{1}=\Psi^{U}$. Therefore, having obtained a solution $(U,\Psi^{U})$ to  problem (\ref{eqn:SBURGERS0}) through the analysis of Section \ref{sec:BRBE}, we immediately obtain the control procedures needed to drive the system to the desired final state as $U_{1}=\Psi^{U}$ and $U_{\,2}=\Psi^{U}_{x},$ and we characterize the initial state we need to start from by $p_{0}(x)=U(0,x)$.

\subsection{Pricing a contingent claim}
We consider a financial market on a finite time-horizon $[0,T]$
consisting of a money market
and a stock, whose prices $S_{0}(\cdot)$ and $S(\cdot)$ evolve according to the SDEs
\begin{align}\label{def:S0andS}
\qquad\qquad dS_{0}(t)&=r(t)S_{0}(t)dt,\qquad 0<t\leq T,\qquad S_{0}(0)=1, \quad \text{and}\phantom{\Big[}\\
dS(t)&=S(t)\big[\mu(t)dt+\sigma(t)dW(t)\big],\quad 0<t\leq T,\quad S(0)=s\geq 0,\nonumber
\end{align}
respectively. The interest rate $r(\cdot)\geq 0,$ the instantaneous rate of return  $\mu(\cdot),$
 and the volatility $\sigma(\cdot)$ are taken to be bounded,
  $ \mathbb{F}$-progressively measurable random processes with $\sigma(\cdot)\neq 0$.

  In this market, we consider an economic agent with an initial endowment $Y(0)$, who at any time $t\in[0,T]$ invests a proportion
$\pi(t)$ of his wealth $Y(t)$ in the stock, saves the remaining amount
$[1-\pi(t)]Y(t)$ in the money market, consumes with a predetermined rate $C(t)\geq 0$, and is obliged to pay taxes
with rate $L(t)\geq 0$. Furthermore, the tax regulation of the market mandates payments proportional
to the size of the wealth $Y(\cdot)$ with respect to the interest rate $r(\cdot)$, but also applies in favor of investing to the stock by offering tax alleviation analogous to the size of the investment $\pi(\cdot)$ and to the undertaken risk captured by the risky asset's volatility $\sigma(\cdot);$ in particular, the tax rate is given by
$$ L(t)=-\pi(t)\sigma^2(t) Y^2(t)+r(t)Y(t).$$
Therefore, in accordance with the market dynamics of (\ref{def:S0andS}),  the \emph{wealth process} $Y(\cdot)\equiv
  Y^{\pi,C}(\cdot),$
corresponding to the portfolio-consumption pair $(\pi,C),$
is the solution of the following SDE
\begin{align}\label{eqn:SDEwealthY}
\qquad\qquad dY(t) & =\big(1-\pi(t)\big)r(t)Y(t)dt+\pi(t)Y(t)\Big[\mu(t)dt+\sigma(t)dW(t)\Big]\nonumber    \\
&\quad\ -C(t)dt-L(t)dt\nonumber    \\
& =\bigg\{\pi(t)\big(\mu(t)-r(t)\big)Y(t)+\pi(t)\sigma^2(t) Y^2(t)-C(t)\bigg\}dt\nonumber    \\
&\quad\ +\pi(t)\sigma(t)Y(t)dW(t).
\end{align}
Of course, all the preceding decisions taken by either the agent or the government should depend on the information available up to $t$ and not anticipate the future, thus
the \emph{portfolio strategy}
$\pi: [0,T]\times \Omega \rightarrow
\mathbb{R}$ and the \emph{consumption strategy}
 $C : [0,T]\times \Omega \rightarrow [0,\infty)$ are assumed to
 be $\mathbb{F}$-progressively measurable processes, and in addition
 verify the technical integrability condition
 $\int_{0}^{T}\big{(}C(t)+ \pi^{2}(t)\big{)}dt<\infty $, almost surely.

Let us now broach in this market the issue of pricing a contingent claim which depends on the volatility of the stock; that is, an $\mathcal{F}(T)$-measurable random
variable $\xi\geq 0$, satisfying proper integrability conditions, of the form $\xi=p\big(X(T)\big)$, where $X(\cdot)$ is the forward process of (\ref{eqn:FBSDESBred}) and $p:\mathbb{R}\times \Omega\rightarrow [0,\infty)$ is an $\mathcal{F}_{T}$-measurable random
field. This random amount $\xi$ represents a liability for its seller that has to be covered
with the smallest amount $Y(0)$ of initial funds at time $t=0$ and the right trading strategy $\pi(\cdot)$ during the interval $[0,T],$ so that
the corresponding wealth process  $Y(t)$ is positive
for all $t\in[0,T]$ and at the end of the time-horizon $Y(T)= p\big(X(T)\big)$ holds without risk.
In other words, to hedge and determine the fair price $Y(0)$ of the contingent claim at time $t=0$ it suffices to find the solution pair $(Y,\pi)$
of the BSDE (\ref{eqn:SDEwealthY}) subject to the previous terminal condition such that $Y(\cdot)>0$.

Setting
  \begin{equation}\label{def:exZ}
  Z(t)\triangleq\pi(t)\sigma(t)Y(t),\qquad 0\leq t< T,
  \end{equation}
  and eliminating the portfolio $\pi(\cdot)$ in the drift term of (\ref{eqn:SDEwealthY}), we get the BSDE
\begin{align*}
dY(t)&=\Big[\sigma(t)Z(t)Y(t)+m(t)Z(t)-C(t)\Big]dt+Z(t)dW(t),\qquad 0\leq t< T,\\
Y(T)&= p\big(X(T)\big),
\end{align*}
where
$$m(\cdot)\triangleq\frac{\mu(\cdot)-r(\cdot)}{\sigma(\cdot)}$$
is the well known relative market of risk process.
Due to Proposition \ref{prop:BSPDEANDFBSDEU}, its solution is given by
the pair $(Y,Z)$ of (\ref{def:triplBSB}), where $(U,\Psi^{U})$ is of class $ C_{\mathbb{F}}\big([0,T];\mathbb{L}^2(\Omega;C^3(\mathbb{R}))\big)\times \mathbb{L}_{\mathbb{F}}^2\big(0,T;C^2(\mathbb{R})\big)$ and solves the BSPDE
\begin{align}\label{eqn:Upric}
 \quad dU&=\bigg[-\frac{1}{2}\;    \sigma^2(t)U_{xx}+\sigma^{2}(t)UU_{x}+\sigma(t)U\Psi^{U}+m(t)\sigma(t)U_{x}-\sigma(t)\Psi_{x}^{U}\nonumber\\
 &\ \ \, \ \ \ +m(t)\Psi^{U}-C(t)\bigg]dt+\Psi^{U}dW(t),\quad 0\leq t<T,\quad x\in \mathbb{R},\\
U(T,x)&=p(x),\qquad x\in \mathbb{R}.\nonumber
 \end{align}
 In turn, Theorem \ref{thm:SCH} provides the solution of this BSPDE via the generalized Cole-Hopf transformation (\ref{eqn:SCH}), in terms of a triplet  $(V,\Psi^{V},\Psi^{\Psi^{V}})\in C_{\mathbb{F}}\big([0,T];\mathbb{L}^2(\Omega;C^4(\mathbb{R}))\big)\times  C_{\mathbb{F}}\big([0,T];\mathbb{L}^2(\Omega;C^3(\mathbb{R}))\big)\times \mathbb{L}_{\mathbb{F}}^2\big(0,T;C^2(\mathbb{R})\big)$ that solves the linear BSPDE
 \begin{align*}
dV&=\left\{-\frac{1}{2}\,\sigma^2(t) V_{xx}+\sigma(t)m(t)V_{x}-\sigma(t)\Psi_{x}^{V}+m(t)\Psi^{V}+C(t)xV\right\}dt\nonumber\\
&\ \ \ \ \ \   +\Psi^{V}dW(t),\quad 0\leq t<T,\quad x\in \mathbb{R},\\
V(T,x)&=q(x),\qquad x\in \mathbb{R}\nonumber
\end{align*}
for $\bar{c}(\cdot)=0$, subject to $V(\, \cdot\, ,\cdot)>0,$
\begin{align*}
d\Psi^{V}=\bigg\{&-\frac{1}{2}\;\sigma^2(t)\Psi_{xx}^{V}
+\Big[\sigma(t)m(t)+\Psi^{\sigma}(t)\Big]\Psi_{x}^{V}\nonumber\\
&+\left(\frac{A^{\sigma}(t)}{\sigma(t)}-m(t)\frac{\Psi^{\sigma}(t)}{\sigma(t)}+C(t)x
-\frac{(\Psi^{\sigma}(t))^2}{\sigma^2(t)}\right)\Psi^{V}-\sigma(t)\Psi_{x}^{\Psi^{V}}\\
&+\left(\frac{\Psi^{\sigma}(t)}{\sigma(t)}+m(t)\right)\Psi^{\Psi^{V}}\bigg\}dt+\Psi^{\Psi^{V}}dW(t)\qquad \text{on}\qquad [0,T)\times\mathbb{R},\nonumber
\end{align*}
and (\ref{eqn:relatpandq}). Then, the hedging portfolio process $\pi(\cdot)$ follows immediately from (\ref{def:exZ}).

Finally, we shall illustrate simple and explicit solutions of the above system in the case of constant model coefficients $r$, $\mu$, $\sigma$, so as $m$, and zero consumption, i.e., $C(\cdot)\equiv 0$.
In light of Examples \ref{ex:SolBSPDEwithC} and \ref{ex:SolBSPDEwithC2}, we postulate strictly positive solutions of the form
\[ V(t,x)\triangleq e^{\alpha\, W(t)} \qquad\text{or}\qquad V(t,x)\triangleq \beta\, e^{ W(t)}\]
for given constants $\alpha$, $\beta>0$, and take $\sigma<0$ so that (\ref{eqn:SCH}) establishes positive solutions
for the BSPDE (\ref{eqn:Upric}), respectively. Then, in order for the corresponding $2\times 2$ systems in these examples to hold, it suffices to
select  the coefficients of the model such that
\[ m(\cdot)= \frac{1}{2}\,\alpha \qquad\text{or}\qquad  m(\cdot)= \frac{1}{2},\]
and from (\ref{eqn:relatpandq}) \[ p(\cdot)= -\frac{\alpha}{\sigma}>0 \qquad\text{or}\qquad p(\cdot)=-\frac{1}{\sigma}>0, \]
respectively.

\section{Acknowledgements}
We are grateful to Professor Ioannis Karatzas for his valuable comments. We wish also to thank an anonymous referee
for reading thoroughly and making constructive suggestions on the paper. This work was supported by Athens University of Economics and Business under the Support Program for Basic Research (PEVE) 2010-2011.

\end{document}